\documentclass[11pt]{article}
\usepackage[margin=0.9in]{geometry}
\usepackage{float}
\usepackage{hyperref}
\usepackage{titling}
\usepackage{blindtext}
\usepackage{amsmath}
\usepackage{amssymb}
\usepackage{mathtools}
\usepackage[numbers]{natbib}
\usepackage{tikz}
\usepackage{graphicx}
\usepackage{siunitx}
\usepackage{gensymb}
\usepackage{url}
\usepackage{mathrsfs}
\usepackage{amsthm}
\usepackage{cleveref}
\usepackage{tkz-euclide}
\usepackage{titlesec}
\usepackage{url, verbatim}
\usepackage{enumerate}
\usepackage{pgfplots}
\usepackage{enumitem}
\usetikzlibrary{calc,patterns,quotes,shapes,arrows,through,intersections,decorations.pathreplacing}
\allowdisplaybreaks

\newtheorem{thm}{Theorem}[section]
\newtheorem{cor}[thm]{Corollary}
\newtheorem{lem}[thm]{Lemma}
\newtheorem{prop}[thm]{Proposition}

\numberwithin{equation}{section}

\newcommand{\mathr}{\mathbb{R}}

\newcommand{\ubox}{\overline{\text{dim}}_\text{B}}
\newcommand{\lbox}{\underline{\text{dim}}_\text{B}}
\newcommand{\boxd}{\text{dim}_\text{B}}
\newcommand{\aso}{{\text{dim}}_\text{A}}

\newcommand{\low}{{\text{dim}}_\text{L}}
\newcommand{\haus}{{\text{dim}}_\text{H}}

\newcommand{\asospec}{{\text{dim}}^{\theta}_\text{A}}
\newcommand{\lowspec}{{\text{dim}}^{\theta}_\text{L}}

\newcommand{\pmax}{p_{\max}}

\newcommand{\rj}{r_{j}(\xi)}
\newcommand{\rjj}{r_{j+1}(\xi)}
\newcommand{\rl}{r_{l}(\xi)}
\newcommand{\rll}{r_{l+1}(\xi)}
\newcommand{\nj}{n_{j}(\xi)}

\newcommand{\petal}{p(\omega)}
\newcommand{\size}[1]{\vert #1 \vert}

\renewcommand{\epsilon}{\varepsilon}

\renewcommand{\geq}{\geqslant}
\renewcommand{\leq}{\leqslant}

\definecolor{lightgray}{rgb}{0.83, 0.83, 0.83}

\title{Assouad type dimensions of parabolic   Julia sets}
\author{Jonathan M. Fraser and Liam Stuart \\
The University of St Andrews, Scotland}

\newlength{\bibitemsep}
\newlength{\bibparskip}
\let\oldthebibliography\thebibliography
\renewcommand\thebibliography[1]{%
  \oldthebibliography{#1}%
  \setlength{\parskip}{\bibitemsep}%
  \setlength{\itemsep}{\bibparskip}%
}

\date{}

\begin{document}
\pagenumbering{arabic}
\maketitle
\begin{abstract}
We prove that the Assouad dimension of a parabolic Julia set is $\max\{1,h\}$ where $h$ is the Hausdorff dimension of the Julia set.  Since $h$ may be strictly less than 1, this provides examples where the Assouad and Hausdorff dimensions are distinct.  The box and packing dimensions of the Julia set are also known to coincide with $h$ and, moreover, $h$ can be characterised by a topological pressure function.  The distinctive behaviour of the Assouad dimension invites further analysis of the `Assouad type dimensions', including the lower dimension and the Assouad and lower spectra.  We derive formulae for all of the Assouad type dimensions for parabolic Julia sets and the associated $h$-conformal measure. Further, we show that if a Julia set has a Cremer point, then the Assouad dimension is 2.\\
 

 \end{abstract}

\tableofcontents

\section{Introduction}\label{Introduction}

The Julia set of a rational map of the extended complex plane is typically a beautiful fractal with intricate geometric properties.   Our main case of interest is when the rational map has a rationally indifferent periodic  point (a parabolic point) but the Julia set contains no critical points (that is, we consider \emph{parabolic Julia sets}).  That said, we also consider Cremer Julia sets which do not fall into this class.  The Hausdorff, box and packing dimensions of a parabolic  Julia set are known to coincide and are given by the  smallest  zero   of the topological pressure, which we denote by $h$, see \cite{DU2}.  The Assouad dimension is a notion of growing relevance in fractal geometry and dimension theory of dynamical systems and up until now the Assouad dimension of Julia sets has not been considered.  Our main result is that, unlike the dimensions discussed above,  the Assouad dimension of a parabolic Julia set  is \emph{not} necessarily equal to $h$. It is instead given by $\max\{1,h\}$ and, since $h$ may be strictly less than 1, this provides examples where the Assouad and Hausdorff dimensions are distinct.  Moreover, it is known that $h<2$ and so the Assouad dimension is also strictly less than 2.  Therefore, our result can be viewed as a refinement of the (known) result that parabolic Julia sets are porous, see \cite[Theorem 1.4]{LG}.   On the other hand, if the Julia set contains a non-linearisable \emph{irrationally} indifferent fixed point (a Cremer point), then we show that its Assouad dimension is 2.

The distinctive behaviour of the Assouad dimension of parabolic Julia sets invites further analysis of the Assouad type dimensions.  We derive formulae for the Assouad and  lower dimensions as well as  the Assouad and lower spectra, both of the Julia set itself and the associated $h$-conformal measure.   These results shed some new light on the `Sullivan dictionary' in the context of dimension theory, see \cite{geomded, stuartsurvey}.  The Assouad and lower spectra were introduced    in \cite{FYu} and provide an `interpolation' between the box dimension and the Assouad and lower dimensions, respectively.    The motivation for the introduction of these `dimension spectra' was to gain a more nuanced understanding of fractal sets than that provided by the dimensions considered in isolation.  This is already proving a fruitful programme with applications emerging in a variety of  settings including problems in  harmonic analysis, see work of Anderson, Hughes, Roos, and Seeger \cite{AHRS, RS}, which uses the Assouad spectrum to study spherical maximal functions.

Our proofs use a variety of techniques.    We take some inspiration from the paper \cite{Fr1}, which gave the Assouad dimension of Kleinian limit sets with parabolic points.  That said, parallels in the strategy only go so far, partly due to the lack of   understanding  of the `hidden 3-dimensional geometry' of Julia sets, see \cite{MC1, MC3}.  We also take  inspiration from the papers \cite{DU4,  SU1, SU2}, where ideas from Diophantine approximation are applied in the context of conformal dynamics. Our key technical tool is the global measure formula from \cite{SU1} which describes the conformal measure of an arbitrary ball in terms of parabolic fluctuations. We also require a quantitative version of the Leau--Fatou flower theorem (see Lemma \ref{qflower}), which describes the geometry of the Julia set near parabolic points. Interestingly, we only require the  flower theorem to study the \emph{lower} dimension and spectrum, emphasising that the lower dimension, although the natural dual to the Assouad dimension,  does not always yield to dual arguments.

For notational convenience we write $A \lesssim B$ if there exists a uniform constant $C \geq 1$ such that $A \leq CB$, and $A \gtrsim B$ if $B \lesssim A$. We write $A \approx B$ if $A \lesssim B$ and $B \lesssim A$.  The important thing here is the uniform nature of the constant $C$, which is only  allowed to depend on parameters which   may be considered fixed throughout the paper, such as the rational map $T$. They also may depend on the interpolation parameter $\theta$, since we may fix it before proving a particular result. But the implicit constants will never depend on variables, such as scales $r>0$ or points $x \in \mathbb{C}$.  Therefore, if $A,B>0$ are themselves constant, it trivially holds that $A \approx B$, but the true value of the notation is when $A$ and $B$ vary with respect to a parameter which the constant $C$ may not depend on.  For example, the number of primes less than $N$ is $\lesssim \frac{N}{\log N}$ by the prime number theorem (with an implicit constant independent of the variable $N$).

\section{Definitions and background}\label{Prelims}
\subsection{Dimensions of sets and measures}
\label{DimPrelims}
We recall the key notions from fractal geometry and dimension theory which we will use throughout the paper.  For a more in-depth treatment see the books \cite{BP, FK} for background on Hausdorff and box dimensions, and \cite{Fr2} for Assouad type dimensions. Julia sets will be subsets of the Riemann sphere  $\hat{\mathbb{C}} = \mathbb{C} \cup \{\infty\}$.  However, by a standard reduction we will assume that the Julia sets are bounded subsets of the complex plane $\mathbb{C}$, which we  identify with $\mathbb{R}^2$.  Therefore, it is convenient to recall dimension theory in Euclidean space only. 

Let $F \subseteq \mathbb{R}^d$.   We write $\haus F$, $\ubox F$ and $\lbox F$  for the \emph{Hausdorff, upper} and \emph{lower box dimensions} of $F$, respectively,  but refer the reader to \cite{BP, FK} for the precise definitions since we do not use them directly.  When the upper and lower box dimensions coincide we write $\boxd F$ for the common value, simply referring to the \emph{box dimension}.  We write 
\[
\size{F} = \sup_{x,y \in F} |x-y| \in [0,\infty]
\]
 to denote the diameter of  $F$.  Given $r>0$, we write $N_r(F)$ for the smallest number of open balls of radius $r$ required to cover $F$. We write $M_r(F)$ to denote the largest cardinality of a packing of $F$ by closed balls of radius $r$ centred in $F$.  In what follows, it is easy to see that  replacing  $N_r(F)$   by $M_r(F)$  yields an equivalent definition and so we sometimes switch between minimal coverings and maximal packings in our arguments. This is standard in dimension theory.

The \textit{Assouad dimension} of $F \subseteq \mathbb{R}^d$ is defined by
\begin{align*}
\aso F = \inf \Bigg\{ s \geq 0 \mid  \exists C>0 \ : \ \forall \ 0<r<R<|F| \  :  \ \forall x \in F \  :  \ 
N_r(B(x,R) \cap F) \leq C \left(\frac{R}{r} \right)^{s} \Bigg\}. 
\end{align*}
Similarly, the \textit{lower dimension} of $F$ is defined by
\begin{align*}
\low F = \sup \Bigg\{ s \geq 0 \mid  \exists C>0 \ : \ \forall \ 0<r<R<|F|  \  :  \ \forall x \in F \  :  \   
N_r(B(x,R) \cap F) \geq C \left(\frac{R}{r} \right)^{s} \Bigg\}
\end{align*}
provided $|F| >0$ and otherwise it is 0. Importantly, for compact $F$,
\[
\low F \leq \haus F \leq \lbox F \leq \ubox F \leq \aso F.
\]
The Assouad and lower spectrum, introduced in \cite{FYu}, interpolate between the box dimensions and the Assouad and lower dimensions in a meaningful way.  They provide a parametrised family of dimensions by fixing the relationship between the two scales $r<R$ used to define Assouad and lower dimension.   Studying the dependence on the parameter within this family  thus yields finer and more nuanced information about the local structure of the set.  For example, one may understand which scales `witness' the behaviour described by the Assouad and lower dimensions.  For $\theta \in (0,1)$, the \textit{Assouad spectrum} of $F$ is given by
\begin{align*}
\asospec F = \inf \Bigg\{ s \geq 0 \mid  \exists C>0 \ : \ \forall \ 0<r<1 \  :  \ \forall x \in F \  :  \ 
N_r(B(x,r^{\theta}) \cap F) \leq C \left(\frac{r^{\theta}}{r} \right)^{s} \Bigg\} 
\end{align*}
and the \textit{lower spectrum} of $F$ by
\begin{align*}
\lowspec F = \sup \Bigg\{ s \geq 0 \mid  \exists C>0 \ : \ \forall \ 0<r<1 \  :  \ \forall x \in F \  :  \  
N_r(B(x,r^{\theta}) \cap F) \geq C \left(\frac{r^{\theta}}{r} \right)^{s} \Bigg\}. 
\end{align*}
See \cite{Fr2} for more background and basic properties of the Assouad and lower spectra.  It was shown in \cite{FYu} that, for a bounded set $F \subseteq \mathbb{R}^d$, 
\begin{align}
\ubox F &\leq \asospec F \leq \min\left\{\aso F, \ \frac{\ubox F}{1-\theta}\right\} \label{basicbound}\\
\low F &\leq \lowspec F \leq \lbox F. \nonumber
\end{align}
In particular, $\asospec F \to \ubox F$ as $\theta \to 0$.  Whilst the analogous statement does not hold for the lower spectrum in general, it was proved in \cite[Theorem 6.3.1]{Fr2} that $\lowspec F \to \lbox F$ as $\theta \to 0$ provided $F$ satisfies a strong form of  dynamical invariance.   Whilst the fractals we study are not quite covered by this result, we shall see that this interpolation holds nevertheless. The limits $\lim_{\theta \to 1} \asospec F $ and $\lim_{\theta \to 1} \lowspec F$ are known to exist in general but may not be the Assouad and lower dimensions, respectively.

There is an analogous dimension theory of measures, and the interplay between the dimension theory of fractals and the measures they support is fundamental to fractal geometry, especially in the dimension theory of dynamical systems. Let $\mu$ be a locally finite Borel measure on $\mathbb{R}^d$. The \textit{Assouad dimension} of $\mu$ with support given by $F$ is defined by
\begin{align*}
\aso \mu = \inf \Bigg\{ s \geq 0 \mid  \exists C>0 \ : \ \forall \ 0<r<R< \vert F \vert \  :  \ \forall x \in F \  :  \  \frac{\mu(B(x,R))}{\mu(B(x,r))} \leq C \left(\frac{R}{r} \right)^{s} \Bigg\} 
\end{align*}
and, provided $ |F|> 0$, the \textit{lower dimension} of $\mu$ is given by
\begin{align*}
\low \mu = \sup \Bigg\{ s \geq 0 \mid  \exists C>0 \ : \ \forall \ 0<r<R< \vert F \vert \  :  \ \forall x \in F \  :  \  \frac{\mu(B(x,R))}{\mu(B(x,r))} \geq C \left(\frac{R}{r} \right)^{s} \Bigg\}
\end{align*}
and otherwise it is 0. By convention we assume that $\inf \emptyset = \infty$. The Assouad and lower dimensions of measures  were introduced in \cite{Ka2}, where they were referred to as the upper and lower regularity dimensions, respectively. It is well known (see \cite[Lemma 4.1.2]{Fr2}) that, for a doubling Borel probability measure $\mu$ with support $F \subseteq \mathbb{R}^{d}$, 
\[
\low \mu \leq \low F \leq \aso F \leq \aso \mu.
\]
In our applications, all measures will be doubling measures. For $\theta \in (0,1)$, the \textit{Assouad spectrum} of $\mu$ with support given by $F$ is defined by
\begin{align*}
\asospec \mu = \inf \Bigg\{ s \geq 0 \mid  \exists C>0 \ : \ \forall \ 0<r< \vert F \vert \  :  \ \forall x \in F \  :  \  \frac{\mu(B(x,r^\theta))}{\mu(B(x,r))} \leq C \left(\frac{r^\theta}{r} \right)^{s} \Bigg\} 
\end{align*}
and, provided $|F| > 0$, the \textit{lower spectrum} of $\mu$ is given by
\begin{align*}
\lowspec \mu = \sup \Bigg\{ s \geq 0 \mid  \exists C>0 \ : \ \forall \ 0<r< \vert F \vert \  :  \ \forall x \in F \  :  \  \frac{\mu(B(x,r^\theta))}{\mu(B(x,r))} \geq C \left(\frac{r^\theta}{r} \right)^{s} \Bigg\} 
\end{align*}
and otherwise it is 0.  It is known (see \cite{FFK} for example) that
\[
\low \mu \leq \lowspec \mu \leq \asospec \mu \leq \aso \mu
\]
 and, if $\mu$ is fully supported on a closed set $F$, then 
\[
\lowspec \mu \leq \lowspec F \leq \asospec F \leq \asospec \mu.
\]
 The \textit{upper box dimension} of $\mu$  with support given by $F$ is defined by
\begin{align*}
\ubox \mu = \inf \Big\{ s \geq 0 \mid  \exists C>0 \ : \ \forall \ 0<r< \vert F \vert \  :  \ \forall x \in F \  :  \  \mu(B(x,r)) \geq Cr^{s} \Big\} 
\end{align*}
and the \textit{lower box dimension} of $\mu$ is given by
\begin{align*}
\lbox \mu = \inf \Big\{ s \geq 0 \mid \exists  C>0 \  : \ \forall  r_0>0 \ : \  \exists \ 0<r<r_0 \ : \ \forall \ x\in F \  :  \  \mu(B(x,r)) \geq Cr^{s} \Big\}.  
\end{align*}
If $\ubox \mu = \lbox \mu$, then we refer to the common value as the \textit{box dimension} of $\mu$, denoted by $\boxd \mu$. These definitions of the box dimension of a measure were introduced only recently  in \cite{FFK}.  There it was also shown that, if the support of $\mu$ is a compact set $F$,  then for $\theta \in (0,1)$,
\[\ubox \mu \leq \asospec \mu \leq \min\left\{\aso \mu, \frac{\ubox \mu}{1-\theta}\right\},\]
\[
\ubox F \leq \ubox \mu
\]
and
\[
\lbox F \leq \lbox \mu.
\]

\subsection{Rational maps and Julia sets}
\label{JuliaPrelims}
 Let $T \colon \hat{\mathbb{C}} \rightarrow \hat{\mathbb{C}}$ be a rational map with $\text{deg}(T) \geq 2$, and write  $J(T)$ to denote  the \textit{Julia set} of $T$, which is equal to the closure of the repelling periodic points of $T$.  The Julia set is closed and $T$-invariant and we assume throughout that it is not the whole sphere.  We may then assume that $J(T)$ is a compact subset of $\mathbb{C}$ by a standard reduction.  If this is not the case, then simply conjugate a point $z \notin J(T)$ to $\infty$ via a M\"obius inversion and then the closedness of the resulting Julia set ensures it lies in a bounded region of $\mathbb{C}$.   This is essentially just choosing a different point on the Riemann sphere to represent the point at infinity.  For a more detailed discussion on the dynamics of rational maps, see \cite{B2,  milnor}. Here and throughout balls are open unless otherwise stated.

A periodic point $\xi \in \hat{\mathbb{C}}$ with period $p$ is said to be \textit{rationally indifferent} or \emph{parabolic} if
$(T^p)^{'}(\xi) = e^{2 \pi i q}$ for some $q \in \mathbb{Q}$. We say that $T$ and $J(T)$ are  \textit{parabolic} if $J(T)$ contains no critical points of $T$, but contains at least one rationally indifferent point.   Recall, that $\xi \in \hat{\mathbb{C}}$  is called a \emph{critical point} if  $T$ fails to be injective in any neighbourhood of $\xi$. This is our main case of interest and we assume throughout that $J(T)$ is parabolic unless we explicitly state otherwise.  We write  $\mathbf{\Omega}$  to denote the finite set of parabolic   points of $T$ and let
\[\mathbf{\Omega}_0 = \{\xi \in \mathbf{\Omega} \mid T(\xi) = \xi, \ T'(\xi) = 1\}.\]
Recall that the set of parabolic points is   finite because every parabolic cycle must attract at least one critical point \cite[Theorem 9.3.2]{B2} and there are only finitely many critical points due to the Riemann--Hurwitz relation \cite[Corollary 2.7.2]{B2}.  As $J(T^{n}) = J(T)$ for every $n \in \mathbb{N}$ and $q$ is rational, we may assume without loss of generality that $\mathbf{\Omega} = \mathbf{\Omega_0}.$ 

It was proven in \cite{DU2} that $h = \haus J(T)$  is given by the smallest zero of the function $t \mapsto  P(T, -t\text{log} \size{T'})$ where $P$ is  the \textit{topological pressure}. A similar result for hyperbolic Julia sets with `smallest' replaced with `only' is often referred to as the \textit{Bowen--Manning--McCluskey formula}, see \cite{BOW, MM}.

We recall, see \cite{DU4, SU1}, that for each $\omega \in \mathbf{\Omega}$, we can find a ball $U_{\omega} = B(\omega,r_{\omega})$ with sufficiently small radius such that on $B(\omega,r_{\omega})$, there exists a unique holomorphic inverse branch $T_{\omega}^{-1}$ of $T$ such that $T_{\omega}^{-1}(\omega) = \omega$.  Further, we can choose $U_{\omega}$ to be so small for each $\omega \in \mathbf{\Omega}$ such that 
\begin{equation} \label{nicecondition}
(U_{\omega} \cup T(U_{\omega})) \cap U_{\omega'} = \emptyset
\end{equation} 
 whenever $\omega, \omega' \in \mathbf{\Omega}$ and $\omega \neq \omega'$. For a parabolic fixed point $\omega \in \mathbf{\Omega}$, the Taylor series of $T$ about $\omega$ is of the form
\[T(z) = z + a(z-\omega)^{\petal + 1} + O\big((z-\omega)^{\petal + 2}\big)\]
for some $a \neq 0$.  We call $\petal \in \mathbb{N}$ the \textit{petal number} of $\omega$, and we write
\begin{equation*}
\pmax = \max\{p(\omega) \mid \omega \in \mathbf{\Omega}\}.
\end{equation*} 
It was proven in \cite{ADU} that $h > \pmax/(1+\pmax)$ and $h<2$. We define the set of \textit{pre-parabolic points} $J_{p}(T)$ by 
\[
J_{p}(T) = \bigcup\limits_{k=0}^{\infty} T^{-k}(\mathbf{\Omega}).
\]
It was proven in \cite[Lemmas 2 and 3]{DU2} that there exists a constant $C>0$ such that to each $\xi \in J(T) \setminus J_{p}(T)$, we can associate a unique maximal sequence (also called the \textit{optimal sequence} at $\xi$) of integers $n_{j}(\xi)$ such that for each $j \in \mathbb{N}$, we have that $T^{n_{j}(\xi)}(\xi) \notin \bigcup_{\omega \in \mathbf{\Omega}} U_{\omega}$ and that the unique holomorphic inverse branch $T_{\xi}^{-n_{j}(\xi)}$ of $T^{n_{j}(\xi)}$ which sends $T^{n_{j}(\xi)}(\xi)$ back to $\xi$ is well-defined on $B(T^{n_{j}(\xi)}(\xi),C)$. We call $J_{r}(T) = J(T) \setminus J_{p}(T)$ the \textit{radial Julia set}.  Following \cite{SU1, SU2}, we define \[r_{j}(\xi) = \vert (T^{\nj}) ' (\xi) \vert^{-1} \]
and call the sequence $(\rj)_{j \in \mathbb{N}}$ the \textit{hyperbolic zoom} at $\xi$. Importantly, for all $\xi \in J_{r}(T)$, $r_{j}(\xi) \to  0$ as $j \to \infty$.  This was implicit in earlier work, e.g.~\cite{DU2,SU1}, but proved explicitly in, for example,   \cite[Lemma 26.2.5 (b)]{noninvertible}. We do not necessarily claim  that  $r_{j}(\xi) $ is decreasing, although this will be true if $|T'(z)| \geq 1$ for all $z \in J(T)$.  

Similarly, for each $\xi \in J_{p}(T)$, we can associate a unique maximal sequence (also called the \textit{optimal terminating sequence} at $\xi$) of integers $n_{j}(\xi)$ such that for each $j \in \{1,2,...,l(\xi)\}$ for some $l(\xi) \geq 1$, we have that $T^{n_{j}(\xi)}(\xi) \notin \bigcup_{\omega \in \mathbf{\Omega}} U_{\omega}$ if $j<l(\xi)$ and $T^{n_{l(\xi)}(\xi)}(\xi)=\omega$ for some $\omega \in \mathbf{\Omega}$ and that for $j<l(\xi)$, the unique holomorphic inverse branch $T_{\xi}^{-n_{j}(\xi)}$ of $T^{n_{j}(\xi)}$ which sends $T^{n_{j}(\xi)}(\xi)$ back to $\xi$ is well-defined on $B(T^{n_{j}(\xi)}(\xi),C)$. We can similarly associate its \textit{terminating hyperbolic zoom}  $(\rj)_{j \in \{1,\dots,l(\xi)\}}$. 

We also require the concept of a \textit{canonical ball}, see \cite{SU2}. Let $\omega \in \mathbf{\Omega}$, and let $I(\omega) = T^{-1}(\omega) \setminus \{\omega\}$. Then for each integer $n \geq 0$, we define the \textit{canonical radius} $r_\xi$ at $\xi \in T^{-n}(I(\omega))$ by \[ r_\xi = \vert (T^n)^{'}(\xi) \vert^{-1}\] and we call $B(\xi,r_\xi)$ the \textit{canonical ball}. We will use the fact that $r_\xi \approx r_{l(\xi)}$, where $r_{l(\xi)}$ is the last element in the terminating hyperbolic zoom at $\xi$. 

Due to work of Sullivan \cite{condynsys} and also Aaronson, Denker and Urba\'nski \cite{ADU, DU1, DU2}, it is known that there exists a unique $h$-conformal measure $m$ supported on $J(T)$, i.e. $m$ is a probability measure such that for each Borel set $F \subset J(T)$ on which $T$ is injective,
\[m(T(F)) = \int_{F} \vert T'(\xi) \vert^h  dm(\xi).\]
In \cite{DU4}, it was shown that $m$ has Hausdorff dimension $h$, and also that the box and packing dimensions of $J(T)$ are equal to $h$. It also follows from, for example, \cite{SU2} that $m$ is exact dimensional and therefore its packing  dimension is also given by $h$.  We derive formulae for  the Assouad and lower dimensions and spectra of $J(T)$ and $m$ in Theorems \ref{juliam} and \ref{julias}.  

It was shown in \cite{SU1} that $m$ has an associated global measure formula which we will make use of throughout.
\begin{thm}[Global Measure Formula]
\label{Global2}
Let $T$ be a parabolic rational map with Julia set $J(T)$ of Hausdorff dimension $h$. Let $m$ denote the associated h-conformal measure supported on $J(T)$. Then there exists a function $\phi:J(T) \times \mathbb{R}^{+} \rightarrow \mathbb{R}^{+}$ such that for all $\xi \in J(T)$ and $0<r<\size{J(T)}$,
\[m(B(\xi,r)) \approx r^h \phi(\xi,r).\]
The values of $\phi$ are determined as follows:

\emph{i)} Suppose  $\xi \in J_r(T)$ has associated optimal sequence $(n_j(\xi))_{j \in \mathbb{N}}$ and hyperbolic zoom $(r_j(\xi))_{j \in \mathbb{N}}$ and $r$ is such that $\rjj \leq r < \rj$ for some $j \in \mathbb{N}$ and   $T^{k}(\xi) \in U_{\omega}$ for all $   n_j(\xi)<k<n_{j+1}(\xi)$ and  some $\omega \in \mathbf{\Omega}$.  Then 
\begin{equation*}
\phi(\xi,r) =  \left\{
        \begin{array}{ll}
      \left(\frac{r}{\rj}\right)^{(h-1)\petal}  & \quad  r\geq  \rj\left(\frac{\rjj}{\rj}\right)^{\frac{1}{1+\petal}} \\
      \left(\frac{\rjj}{r}\right)^{h-1} & \quad  r \leq \rj\left(\frac{\rjj}{\rj}\right)^{\frac{1}{1+\petal}}.
        \end{array}
    \right.
\end{equation*}
On the other hand, if $n_j(\xi)$ and $n_{j+1}(\xi)$ are consecutive integers, then $\phi(\xi,r)  = 1$.

\emph{ii)} Suppose $\xi \in J_p(T)$ has  associated optimal terminating sequence $(n_j(\xi))_{j = 1,\dots,l(\xi)}$ and hyperbolic zoom $(r_j(\xi))_{j =1,\dots,l(\xi)}$. Suppose  $T^{n_{l(\xi)}(\xi)}(\xi) = \omega$ for some $\omega \in \mathbf{\Omega}$. If $r \geq r_{l(\xi)}(\xi)$, the values of $\phi$ are determined as in the radial case, and if $r \leq r_{l(\xi)}(\xi)$, then
\[\phi(\xi,r) = \left(\frac{r}{r_{l(\xi)}(\xi)}\right)^{(h-1)\petal}.\]
\end{thm}

Since we do not necessarily have that $r_{j}(\xi)$ is decreasing, it may be the case that, for a particular $r>0$ and $\xi$,  $\rjj \leq r < \rj$ for more than one $j \in \mathbb{N}$.  In this case, the global measure formula gives different expressions for the measure of $m(B(\xi,r)) $ but these are all comparable up to uniform constants. 

Further, observe that the $p(\omega)$ used in the definition of $\phi$ is well-defined due to \eqref{nicecondition}.  In particular, the forward  orbit of $\xi$ cannot jump from $U_\omega$ to $U_{\omega'}$ with $\omega \neq \omega'$ in one step.  Instead, the orbit will enter some $U_\omega$ and stay there for some number of steps (during which time the parabolic fluctuations in $m(B(\xi,r))$ depend on $\omega$) and then it will leave (at which time  the parabolic fluctuations disappear) before later re-entering another (possibly different) $U_\omega'$. Given this, we may think of $\phi(\xi,r)$ as being defined via a particular $\omega$ in the cases where it is not equal to 1. 

If $J(T)$ contains no parabolic points nor critical points, then it is hyperbolic and
\[\aso J(T) = \low J(T) = \aso m = \low m =\boxd m = h\]
and
\[
\asospec J(T) = \asospec m = \lowspec J(T) = \lowspec m = h
\]
for all $\theta \in (0,1)$. Here $m$ is (again) the unique $h$-conformal measure for $T$ which in this case is Ahlfors regular.  See, for example, \cite[Theorem 4]{condynsys}.

 A final case of interest is when $J(T)$ contains a non-linearisable  \emph{irrationally} indifferent fixed point (a \emph{Cremer point}).  In this case the Jacobian derivative of $T$ at the Cremer point is an irrational rotation and  $T$ is not linearisable in a neighbourhood of the Cremer point.  Such rational maps $T$ are not parabolic, but we consider them below in Theorem \ref{cremer}. Julia sets with Cremer points are notoriously difficult to study and it is conjectured that they should have Hausdorff dimension 2 (even positive Lebesgue measure).  See \cite{cheraghi} for more discussion of this problem and some partial results.

\section{Main results: dimension theory of Julia sets}\label{JuliaResults}

Our first result gives a comprehensive description of the dimension theory of the  $h$-conformal measures described above.

\begin{thm}\label{juliam} Let  $m$ be   the    $h$-conformal measure  associated with a parabolic rational map and let $\theta \in (0,1)$. Then
\begin{enumerate}[label=(\roman*)]
\item  $\normalfont{\boxd} m =  \max\{h,h+(h-1)\pmax\}$.\label{thmm1}

\item $\normalfont{\aso} m = \max\{1,h+(h-1)\pmax\}$.\label{thmm2}

\item $\normalfont{\low} m = \min\{1,h+(h-1)\pmax\}$.\label{thmm3}

\item If $h\leq1$, then 
\[\normalfont{\asospec} m = h+\min\left\{1,\frac{\theta \, \pmax}{1-\theta}\right\}(1-h), \] 

and if $h\geq 1$, then  $\normalfont{\asospec} m = h+(h-1)\pmax.$ \label{thmm4}

\item  If $h\leq1$, then  $\normalfont{\lowspec} m = h+(h-1)\pmax$, and if $h \geq 1$, then 
\[\normalfont{\lowspec} m = h+\min\left\{1,\frac{\theta \, \pmax}{1-\theta}\right\}(1-h).\]
\label{thmm5}
\end{enumerate}
\end{thm}
We prove Theorem \ref{juliam} in Section \ref{juliamproof}.  Many interesting quantitative features of the dimension theory of $m$ can be read from  Theorem \ref{juliam}.  For example, the different dimensions may take on different values in certain explicit cases and the Assouad and lower spectra have a phase transition at $1/(1+\pmax)$ in cases when they are not constant. 

Next we turn our attention to the  Julia set  $J(T)$.  

\begin{thm}\label{julias} Let $T$ be  a parabolic rational map with Julia set $J(T)$ and let $\theta \in (0,1)$. Then
\begin{enumerate}[label=(\roman*)]
\item $\normalfont{\aso} J(T) = \max\{1,h\}$. \label{thmj1}

\item  $\normalfont{\low} J(T) = \min\{1,h\}$.\label{thmj2}

\item  If $h\leq1$, then \[\normalfont{\asospec} J(T) = h+\min\left\{1,\frac{\theta \, \pmax}{1-\theta}\right\}(1-h), \] \label{thmj3}
and if $h\geq 1$, then $\normalfont{\asospec} J(T) = h .$ 

\item  If $h\leq1$, then $\normalfont{\lowspec} J(T) = h$, and if $h \geq 1$, then \[\normalfont{\lowspec} J(T) = h+\min\left\{1,\frac{\theta \, \pmax}{1-\theta}\right\}(1-h).\]\label{thmj4}
\end{enumerate}
\end{thm}
We prove Theorem \ref{julias} in Section \ref{juliasproof}.   Again, there are several intriguing features which can be read from Theorem \ref{julias}.  For example, the Assouad dimension necessarily lies in the half-open interval $[1,2)$ and is distinct from $h$ whenever $h<1$.  The fact that $\normalfont{\aso} J(T) <2$  can be viewed as a refinement of the fact that $J(T)$ is porous, see \cite[Theorem 1.4]{LG}.  Moreover, the Assouad and lower dimensions may differ from the Hausdorff dimension, but yet it is not possible for the three dimensions to be distinct simultaneously. Further,  the Assouad and lower spectra again have a phase transition at $1/(1+\pmax)$ in cases when they are not constant. 

 It is   perhaps noteworthy that in the special case $h=1$ all of the formulae in Theorems \ref{juliam} and  \ref{julias}    reduce to $1$.  This can already be seen from the global measure formula, Theorem \ref{Global2}, because the conformal measure $m$ is Ahlfors regular of dimension 1.  In particular, in the case $h=1$ there is nothing to prove.

The case where the rational map is not parabolic (or hyperbolic)  remains an interesting open programme.  For example, we do not know how to compute the Assouad dimension when the rational map  has a Herman ring or a Siegel disk. McMullen \cite{mcmullensiegel} constructed examples of quadratic polynomials with Siegel disks whose Julia set is porous.  In particular, the Assouad (and Hausdorff) dimension of the Julia set is strictly less than 2 by \cite[Theorem 5.2]{LK}, where it was proved that a set in $\mathbb{R}^d$ is porous if and only if its Assouad dimension is strictly less than $d$. 

 We obtain the following  result for Julia sets with Cremer points. Given that such  Julia sets are conjectured to have Hausdorff dimension 2, one may obtain a weaker conjecture by replacing Hausdorff dimension with a larger notion of fractal dimension. Here we resolve the weakest version of this conjecture by replacing Hausdorff dimension with Assouad dimension.   

\begin{thm} \label{cremer}
If $T$ is a rational map and $J(T)$ contains a Cremer point, then $\normalfont{\aso} J(T)   =2$.
\end{thm}

We prove Theorem \ref{cremer} in Section \ref{cremerproof}.  The proof of Theorem \ref{cremer} is self-contained and does not rely on, e.g. a global measure formula.

\section{Conformal measure: proof of Theorem \ref{juliam} } \label{juliamproof}

\subsection{Box dimension of $m$: proof of Theorem \ref{juliam} part \ref{thmm1}}\label{BoxM}

Clearly $\lbox m \geq \lbox J(T) = h$.  Moreover,  if $h \leq 1$, then  for all  $\xi \in J(T)$ and $r < \size{J(T)}$
\begin{equation*}
m(B(\xi,r)) \approx r^h \phi(\xi,r) \gtrsim r^h
\end{equation*}
giving $\ubox m \leq h$. On the other hand, suppose $h \geq 1$, and let $\xi =\omega \in \Omega \subseteq  J_p(T)$  with $\petal = \pmax$, noting that  $T^{k}(\omega) = \omega$ for all $k \geq 1$.  Then, using Theorem \ref{Global2},  for all sufficiently small $r>0$
\begin{equation*}
m(B(\omega,r)) \lesssim r^{h+(h-1)\pmax}
\end{equation*}
which proves $\lbox m \geq h+(h-1)\pmax$. The upper bound in this case will follow from the more general upper bound for Assouad dimension and is therefore omitted, see Section \ref{AsoM} below.  That is,
 \begin{equation*}
\ubox m \leq \aso m \leq h+(h-1)\pmax.
\end{equation*}
 In all cases we conclude $\lbox m = \ubox m = \max\{h,h+(h-1)\pmax\}$ as required.

\subsection{Assouad dimension of $m$: proof of Theorem \ref{juliam} part \ref{thmm2}}\label{AsoM}

The lower bound will follow from our lower bound for the Assouad spectrum of $m$, see Section \ref{AsospecM}.  Therefore we only need to prove the upper bound.  That is, we show \[\aso m \leq \max\{1,h+(h-1)\pmax\}.\]
We use Theorem \ref{Global2} throughout the proof without mentioning it explicitly.  Consider $0<r<R<\vert J(T) \vert$. Suppose $\xi \in J_r(T)$, with associated optimal sequence $(n_j(\xi))_{j \in \mathbb{N}}$ and hyperbolic zoom $(r_j(\xi))_{j \in \mathbb{N}}$. The case where  $\xi \in J_p(T)$ follows similarly and is omitted. 

Suppose that $\rjj \leq r < R < \rj$ and that $T^{k}(\xi) \in U_{\omega} = B(\omega,r_{\omega})$ for some $\omega \in \mathbf{\Omega}$ for $n_j(\xi)<k<n_{j+1}(\xi)$, and let $r_m = \rj\left({\rjj}/{\rj}\right)^{\frac{1}{1+\petal}}$. 

If $r > r_m$, then
\begin{align*}
\frac{m(B(\xi,R))}{m(B(\xi,r))}  \approx \left(\frac{R}{r}\right)^h \frac{\left(R/\rj\right)^{(h-1)\petal}}{\left(r/\rj \right)^{(h-1)\petal}} 
\leq \left(\frac{R}{r}\right)^{\max\{1,h+(h-1)\pmax\}}.
\end{align*}

If $R < r_m$, then
\begin{align*}
\frac{m(B(\xi,R))}{m(B(\xi,r))}  \approx  \left(\frac{R}{r}\right)^h  \frac{\left(\rjj/R\right)^{h-1}}{\left(\rjj/r\right)^{h-1}} 
= \frac{R}{r}.
\end{align*}

If $r \leq r_m \leq R$, then
\begin{align}\label{form1}
\frac{m(B(\xi,R))}{m(B(\xi,r))}  \approx  \left(\frac{R}{r}\right)^h \frac{\left(R/\rj\right)^{(h-1)\petal}}{\left(\rjj/r\right)^{h-1}}
=  \frac{R^{h+(h-1)\petal}}{r\rj^{(h-1)\petal}\rjj^{h-1}}.
\end{align}
If we assume that $h \geq 1$, then note that 
\begin{align*}
\left(\frac{r}{R} \right)^{h+(h-1)\petal}\frac{m(B(\xi,R))}{m(B(\xi,r))} \approx \frac{ r^{(h-1)(1+\petal)}}{\rj^{(h-1)\petal}\rjj^{h-1}}
\end{align*}
is maximised when $r = r_m$. Therefore
\begin{align*}
\left(\frac{r}{R} \right)^{h+(h-1)\petal}\frac{m(B(\xi,R))}{m(B(\xi,r))} \lesssim  \frac{{r_m}^{(h-1)(1+\petal)}}{\rj^{(h-1)\petal}\rjj^{h-1}}
= \frac{ \rj^{(h-1)(1+\petal)} \left(\frac{\rjj}{\rj} \right)^{h-1}}{\rj^{(h-1)\petal}\rjj^{h-1}} 
= 1,
\end{align*}
which proves that
\begin{align*}
\frac{m(B(\xi,R))}{m(B(\xi,r))} \lesssim \left(\frac{R}{r} \right)^{h+(h-1)\petal} \leq \left(\frac{R}{r} \right)^{h+(h-1)\pmax}.
\end{align*}
Similarly, if $h \leq 1$, then by (\ref{form1})

\begin{align*}
\left(\frac{r}{R} \right)\frac{m(B(\xi,R))}{m(B(\xi,r))} \approx  \frac{R^{(h-1)(1+\petal)}}{\rj^{(h-1)\petal}\rjj^{h-1}}
\end{align*}
is maximised when $R=r_m$. Therefore, as above
\begin{align*}
\left(\frac{r}{R} \right)\frac{m(B(\xi,R))}{m(B(\xi,r))} \lesssim 1
\end{align*}
which proves
\begin{align*}
\frac{m(B(\xi,R))}{m(B(\xi,r))} \lesssim \frac{R}{r}.
\end{align*}
This covers all cases when $\rjj \leq r < R < \rj$.

Now, we consider the case when $\rjj \leq R < \rj$, $\rll \leq r < \rl$, for some $l>j$, $T^{k}(\xi) \in U_{\omega_1}$ if $n_j(\xi)<k<n_{j+1}(\xi)$ and $T^{k}(\xi) \in U_{\omega_2}$ if $n_l(\xi)<k<n_{l+1}(\xi)$ for some $\omega_1,\omega_2 \in \mathbf{\Omega}$. Note that we can assume  in this case that $\rl \leq R$ and $\rjj \geq r$ (otherwise we are covered by the previous case); consequently, $\rl \leq \rj$ and $\rll \leq \rjj$.   \\
Let $r_m = \rj\left({\rjj}/{\rj}\right)^{\frac{1}{1+p(\omega_1)}}$ and $r_n =\rl\left({\rll}/{\rl}\right)^{\frac{1}{1+p(\omega_2)}}$.\\

\noindent \textit{Case 1}: $R>r_m$, $r>r_n$.

We have
\begin{equation}\label{form2}
\frac{m(B(\xi,R))}{m(B(\xi,r))} \approx  \left(\frac{R}{r}\right)^h  \frac{\left(R/\rj\right)^{(h-1)p(\omega_1)}}{\left(r/\rl\right)^{(h-1)p(\omega_2)}}. 
\end{equation}
If $h \leq 1$, then, using $r \leq \rjj$,
\begin{align*}
 \left(\frac{r}{R}\right) \frac{m(B(\xi,R))}{m(B(\xi,r))}  \lesssim \frac{R^{(h-1)(1+p(\omega_1))}}{r^{h-1}\rj^{(h-1)p(\omega_1)}} 
&\leq \frac{r_m^{(h-1)(1+p(\omega_1))}}{\rjj^{h-1}\rj^{(h-1)p(\omega_1)}} \\
&=\frac{\rj^{(h-1)(1+p(\omega_1))}\left(\frac{\rjj}{\rj}\right)^{h-1}}{\rjj^{h-1}\rj^{(h-1)p(\omega_1)}}
= 1,
\end{align*}
 and if $h \geq 1$, then by (\ref{form2})
\begin{align*}
\frac{m(B(\xi,R))}{m(B(\xi,r))}  \lesssim \left(\frac{R}{r}\right)^h \left(\frac{\rl}{r}\right)^{(h-1)p(\omega_2)}
\leq \left(\frac{R}{r} \right)^{h+(h-1)p(\omega_2)} 
\end{align*}
using $\rl \leq R$.\\

\noindent \textit{Case 2}: $R\leq r_m$, $r \leq r_n$.

We have 
\begin{equation}\label{form3}
\frac{m(B(\xi,R))}{m(B(\xi,r))} \approx  \left(\frac{R}{r}\right)^h  \frac{\left(\rjj/R\right)^{h-1}}{\left(\rll/r\right)^{h-1}}.
\end{equation}
If $h \leq 1$, then, using $\rll \leq \rjj$,
\begin{align*}
\frac{m(B(\xi,R))}{m(B(\xi,r))} \approx  \left(\frac{R}{r}\right)  \frac{\rjj^{h-1}}{\rll^{h-1}}
\leq  \frac{R}{r},
\end{align*}
and if $h \geq 1$, then by (\ref{form3})
\begin{equation*}
\frac{m(B(\xi,R))}{m(B(\xi,r))} \lesssim \frac{R^h}{r \rll^{h-1}} 
\end{equation*}
and therefore
\begin{align*}
 \left(\frac{r}{R}\right)^{h+(h-1)p(\omega_2)} \frac{m(B(\xi,R))}{m(B(\xi,r))} &\lesssim \frac{R^{(1-h)p(\omega_2)}}{r^{(1-h)(1+p(\omega_2))} \rll^{h-1}}\\ 
&\leq \frac{\rl^{(1-h)p(\omega_2)}}{r_n^{(1-h)(1+p(\omega_2))} \rll^{h-1}} \\
&= \frac{\rl^{(1-h)p(\omega_2)}}{\rl^{(1-h)(1+p(\omega_2))}\left(\frac{\rll}{\rl}\right)^{1-h} \rll^{h-1}}
=1.
\end{align*}
\textit{Case 3}: $R> r_m$, $r \leq r_n$.

We have
\begin{align}\label{form4}
\frac{m(B(\xi,R))}{m(B(\xi,r))} &\approx  \left(\frac{R}{r}\right)^h  \frac{\left(R/\rj\right)^{(h-1)p(\omega_1)}}{\left(\rll/r\right)^{(h-1)}} 
= \frac{R^{h+(h-1)p(\omega_1)}}{r\rj^{(h-1)p(\omega_1)}\rll^{h-1}}.
\end{align}
If $h \leq 1$, then 
\begin{align*}
 \left(\frac{r}{R}\right) \frac{m(B(\xi,R))}{m(B(\xi,r))} 
 \approx \frac{R^{(h-1)(1+p(\omega_{1}))}}{\rj^{(h-1)p(\omega_1)}\rll^{h-1}} 
 &\leq \frac{\rj^{(h-1)(1+p(\omega_{1}))}\left(\frac{\rjj}{\rj}\right)^{h-1}}{\rj^{(h-1)p(\omega_1)}\rjj^{h-1}} 
= 1.
\end{align*}
If $h \geq 1$ and $p(\omega_1) \geq p(\omega_2)$, then by (\ref{form4})
\begin{align*}
 \left(\frac{r}{R}\right)^{h+(h-1)p(\omega_1)} \frac{m(B(\xi,R))}{m(B(\xi,r))} &\approx \frac{r^{(h-1)(1+p(\omega_{1}))}}{\rj^{(h-1)p(\omega_1)}\rll^{h-1}}\\
&\leq \frac{\rl^{(h-1)(1+p(\omega_{1}))}\left(\frac{\rll}{\rl}\right)^{\frac{(h-1)(1+p(\omega_1))}{1+p(\omega_2)}}}{\rl^{(h-1)p(\omega_1)}\rll^{h-1}} \\
&= \left(\frac{\rl}{\rll}\right)^{(h-1)(1-\frac{1+p(\omega_1)}{1+p(\omega_2)})} 
\leq 1,
\end{align*}
and if $h \geq 1$ and $p(\omega_1) < p(\omega_2)$, then by (\ref{form4}), using $\rl \leq \rj$
\begin{align*}
 \left(\frac{r}{R}\right)^{h+(h-1)p(\omega_2)} \frac{m(B(\xi,R))}{m(B(\xi,r))} &\approx \frac{r^{(h-1)(1+p(\omega_{2}))}R^{(h-1)(p(\omega_1)-p(\omega_2))}}{\rj^{(h-1)p(\omega_1)}\rll^{h-1}}\\
&\leq  \frac{r_n^{(h-1)(1+p(\omega_{2}))}R^{(h-1)(p(\omega_1)-p(\omega_2))}}{\rl^{(h-1)p(\omega_1)}\rll^{h-1}}\\
&= \left(\frac{R}{\rl}\right)^{(h-1)(p(\omega_1)-p(\omega_2))}  
\leq 1.
\end{align*}
\textit{Case 4}: $R\leq r_m$, $r > r_n$.

This gives
\begin{equation}\label{form5}
\frac{m(B(\xi,R))}{m(B(\xi,r))} \approx  \left(\frac{R}{r}\right)^h  \frac{\left(\rjj/R\right)^{h-1}} {\left(r/\rl\right)^{(h-1)p(\omega_2)}}.
\end{equation}
Although we cannot guarantee $r_l(\xi) \leq r_{j+1}(\xi)$, it is true that $r_l(\xi) \lesssim  r_{j+1}(\xi)$ in the case $h \neq 1$, and this can be seen as follows.  Suppose $r_l(\xi) >  r_{j+1}(\xi)$ in which case $r_{l+1}(\xi) <r_n< r\leq  r_{j+1}(\xi) <r_l(\xi)  $ and then, by  two applications of Theorem \ref{Global2} using different `windows',
\begin{equation*}
r_{j+1}(\xi)^h \approx m(B(\xi, r_{j+1}(\xi) ) \approx r_{j+1}(\xi)^h \left(\frac{r_{j+1}(\xi)}{r_l(\xi)}\right)^{(h-1)p(\omega_2)}
\end{equation*}
which gives the claim. If $h < 1$, then, using that $r_l(\xi) \lesssim  r_{j+1}(\xi)$,
\begin{align*}
 \left(\frac{r}{R}\right) \frac{m(B(\xi,R))}{m(B(\xi,r))} \approx \frac{r^{(1-h)(1+p(\omega_2))}}{\rjj^{1-h}\rl^{(1-h)p(\omega_2)}} 
\lesssim \frac{\rl^{(1-h)(1+p(\omega_2))}}{\rl^{1-h}\rl^{(1-h)p(\omega_2)}}
= 1,
\end{align*}
and if $h \geq 1$, then by (\ref{form5})
\begin{align*}
\frac{m(B(\xi,R))}{m(B(\xi,r))} \lesssim  \left(\frac{R}{r}\right)^h \left(\frac{\rl}{r}\right)^{(h-1)p(\omega_2)} 
\leq   \left(\frac{R}{r}\right)^{h+(h-1)p(\omega_2)}.
\end{align*}
In all possible cases, 
\begin{align*}
\frac{m(B(\xi,R))}{m(B(\xi,r))} \lesssim  \left(\frac{R}{r}\right)^{\max\{1, h+(h-1)p(\omega_1), h+(h-1)p(\omega_2)\}} \leq \left(\frac{R}{r}\right)^{\max\{1, h+(h-1)\pmax\}}
\end{align*}
which proves the desired upper bound.


\subsection{Lower dimension of $m$: proof of Theorem \ref{juliam} part \ref{thmm3}}\label{LowM}


The upper bound will follow from our upper  bound for the lower spectrum of $m$, see Section \ref{LowspecM}.  Therefore we only need to prove the lower bound. However, the lower bound for $\low m$  can be proved by a completely analogous argument to that given for the upper bound for $\aso m$ in the
previous section, and so we leave it for the reader.  In particular, the roles of  $h \leq 1$ and $h \geq 1$ are reversed which reverses many of the inequalities.

\subsection{Assouad spectrum of $m$: proof of Theorem \ref{juliam} part \ref{thmm4}}\label{AsospecM}

\subsubsection{When $h < 1$}\label{AsospecM1} 

The lower bound here  follows from the lower bound for the Assouad spectrum of $J(T)$, see Section \ref{AsospecJ1}.  Therefore it remains to prove the upper bound, that is, we show \[\asospec m \leq h+\min\left\{1,\frac{\theta\pmax}{1-\theta}\right\}(1-h) .\]
The case when $\theta \geq {1}/{(1+\pmax)}$ follows easily, as 
\[\asospec m \leq \aso m \leq 1,\] 
so we assume $\theta < {1}/{(1+\pmax)}$. Let $\xi \in J_r(T)$, and $r>0$ be small and assume that $\rjj \leq r^\theta < \rj$, $\rll \leq r < \rl$, with $l \geq j$, $T^{k}(\xi) \in U_{\omega_1}$ for $n_j(\xi) <k<n_{j+1}(\xi)$ and $T^{k}(\xi) \in U_{\omega_2}$ for $n_l(\xi) <k<n_{l+1}(\xi)$ for some $\omega_1,\omega_2 \in \mathbf{\Omega}$.  The case $\xi \in J_p(T)$ is similar and omitted. Note that $m(B(\xi,R)) \gtrsim r^h$ and we may assume $\rj \lesssim 1$. To justify this latter fact, observe that
\begin{equation} \label{zoombounded}
\min\{r_j(\xi), 1/r_{j+1}(\xi)\} \lesssim 1.
\end{equation}
In particular, for sufficiently small $r$, $\rjj \leq r^\theta < \rj$ ensures that $\rj \lesssim 1$.  To establish \eqref{zoombounded}, suppose  $\min\{r_j(\xi), 1/r_{j+1}(\xi)\} \geq K > 1$.  Then $r_{j+1}(\xi) <1 < r_{j}(\xi)$ and Theorem \ref{Global2} implies
\[
1 \approx m(B(\xi,1)) \approx \phi(\xi, 1) 
\]
but  the right hand side is not comparable to 1 as $K \to \infty$, recalling that  $h \neq 1$.

Let $r_m = \rj\left({\rjj}/{\rj}\right)^{\frac{1}{1+p(\omega_1)}}$. If $r^\theta > r_m$, then by  Theorem \ref{Global2}
\begin{align*}
\frac{m(B(\xi,r^\theta))}{m(B(\xi,r))} \lesssim \left(\frac{r^\theta}{r}\right)^h \left(\frac{r^\theta}{\rj}\right)^{(h-1)p(\omega_1)} 
\leq  \left(\frac{r^\theta}{r}\right)^h \left(\frac{r^\theta}{\rj}\right)^{(h-1) \pmax} &\lesssim \left(\frac{r^\theta}{r}\right)^h r^{\theta(h-1)\pmax} \\
&=\left(\frac{r^\theta}{r}\right)^{h+\frac{\theta\pmax}{1-\theta}(1-h)},
\end{align*}
and if $r^\theta \leq r_m$, then by Theorem \ref{Global2}
\begin{align*}
\frac{m(B(\xi,r^\theta))}{m(B(\xi,r))}  \lesssim  \left(\frac{r^\theta}{r}\right)^h \left(\frac{\rjj}{r^\theta}\right)^{h-1}
\leq \left(\frac{r^\theta}{r}\right)^h \left(\frac{r^\theta}{\rj}\right)^{(h-1) \pmax} &\lesssim \left(\frac{r^\theta}{r}\right)^h r^{\theta(h-1)\pmax} \\
&=\left(\frac{r^\theta}{r}\right)^{h+\frac{\theta\pmax}{1-\theta}(1-h)}.
\end{align*}
In either case, 
\begin{equation*}
\frac{m(B(\xi,r^\theta))}{m(B(\xi,r))} \lesssim  \left(\frac{r^\theta}{r}\right)^{h+\frac{\theta\pmax}{1-\theta}(1-h)},
\end{equation*}
which proves \[\asospec m \leq h+\frac{\theta\pmax}{1-\theta}(1-h)\] as required.

 \subsubsection{When $h \geq 1$}\label{AsospecM2}   We show $\asospec m = h+(h-1)\pmax .$ This follows easily, since 
\[
h+(h-1)\pmax = \ubox m \leq \asospec m \leq \aso m \leq h+(h-1)\pmax.
\]

\subsection{Lower spectrum of $m$: proof of Theorem \ref{juliam} part \ref{thmm5}}\label{LowspecM}
\subsubsection{When $h \leq 1$}\label{LowspecM1}
 We show $\lowspec m = h+(h-1)\pmax.$ Note that 
\begin{equation*}
\lowspec m \geq \low m \geq h+(h-1)\pmax
\end{equation*}
and so we need only prove the upper bound. To do this, let $\xi =\omega \in \Omega \subseteq  J_p(T)$  with $\petal = \pmax$, noting that  $T^{k}(\omega) = \omega$ for all $k \geq 1$.   Then, using Theorem \ref{Global2},  for all sufficiently small $r>0$
\begin{align*}
\frac{m(B(\omega,r^\theta))}{m(B(\omega,r))} \lesssim \left(\frac{r^\theta}{r}\right)^h \left(\frac{r^\theta}{r}\right)^{(h-1)\petal} 
=   \left(\frac{r^\theta}{r}\right)^{h+(h-1)\pmax}
\end{align*}
which proves $\lowspec m \leq h+(h-1)\pmax$, as required.

\subsubsection{When $h > 1$}\label{LowspecM2} 

The upper bound here  follows from the upper bound for the lower spectrum of $J(T)$, see Section \ref{LowspecJ2}.  Therefore it remains to prove the lower bound, that is, we show \[\lowspec m \geq h+\min\left\{1,\frac{\theta\pmax}{1-\theta}\right\}(1-h) .\]
The case when $\theta \geq {1}/{(1+\pmax)}$ follows easily, as  $\lowspec m \geq \low m \geq 1,$ so we assume $\theta < {1}/{(1+\pmax)}$.
Let $\xi \in J_r(T)$, and $r>0$ be small and assume that $\rjj \leq r^\theta < \rj$, $\rll \leq r < \rl$, with $l \geq j$, $T^{k}(\xi) \in U_{\omega_1}$ for $n_j(\xi) < k < n_{j+1}(\xi)$ and $T^{k}(\xi) \in U_{\omega_2}$ for $n_l(\xi) < k < n_{l+1}(\xi)$ for some $\omega_1,\omega_2 \in \mathbf{\Omega}$.   The case $\xi \in J_p(T)$ is similar and omitted. Note that $m(B(\xi,R)) \lesssim r^h$ and we may assume $\rj \lesssim 1$.  This latter fact can be justifed by \eqref{zoombounded} and the subsequent discussion in Section \ref{AsospecM1}.

Let $r_m = \rj\left({\rjj}/{\rj}\right)^{\frac{1}{1+p(\omega_1)}}$. If $r^\theta > r_m$, then by Theorem \ref{Global2}
\begin{align*}
\frac{m(B(\xi,r^\theta))}{m(B(\xi,r))} \gtrsim \left(\frac{r^\theta}{r}\right)^h \left(\frac{r^\theta}{\rj}\right)^{(h-1)p(\omega_1)} 
\geq  \left(\frac{r^\theta}{r}\right)^h \left(\frac{r^\theta}{\rj}\right)^{(h-1) \pmax} &\gtrsim \left(\frac{r^\theta}{r}\right)^h r^{\theta(h-1)\pmax} \\
&=\left(\frac{r^\theta}{r}\right)^{h+\frac{\theta\pmax}{1-\theta}(1-h)},
\end{align*}
and if $r^\theta \leq r_m$, Theorem \ref{Global2} gives
\begin{align*}
\frac{m(B(\xi,r^\theta))}{m(B(\xi,r))}  \gtrsim  \left(\frac{r^\theta}{r}\right)^h \left(\frac{\rjj}{r^\theta}\right)^{h-1} 
\geq  \left(\frac{r^\theta}{r}\right)^h \left(\frac{r^\theta}{\rj}\right)^{(h-1) \pmax} &\gtrsim \left(\frac{r^\theta}{r}\right)^h r^{\theta(h-1)\pmax} \\
&=\left(\frac{r^\theta}{r}\right)^{h+\frac{\theta\pmax}{1-\theta}(1-h)}.
\end{align*}
In either case, 
\begin{equation*}
\frac{m(B(\xi,r^\theta))}{m(B(\xi,r))} \gtrsim  \left(\frac{r^\theta}{r}\right)^{h+\frac{\theta\pmax}{1-\theta}(1-h)},
\end{equation*}
which proves \[\lowspec m \geq {h+\frac{\theta\pmax}{1-\theta}(1-h)}\] as required.

\section{Parabolic Julia sets: proof of Theorem \ref{julias}} \label{juliasproof}

\subsection{Preliminary estimates}\label{JuliaPre}
The following lemma will be used to  count canonical balls of certain sizes. 
\begin{lem}\label{CountCanon}
Let  $\omega \in \mathbf{\Omega}$ and $I := \bigcup\limits_{n \geq 0} T^{-n}(I(\omega))$.  Let 
$\xi \in J(T)$ and $R>r>0$. For $R$ sufficiently small, 
\begin{equation*}
\sum_{\substack{c(\omega) \in I \cap B(\xi,R) \\ R > r_{c(\omega)} \geq r} } r_{c(\omega)}^h \lesssim \text{\normalfont{log}}(R/r) \ m(B(\xi,R)).
\end{equation*}
 \end{lem}
\begin{proof}
By \cite[Theorem 3.1]{SU2}, there exists a constant $\kappa > 0$ dependent only on $T$ such that   for sufficiently small $r>0$, 
\begin{equation*}
J(T) \subseteq \bigcup_{\substack{c(\omega) \in I  \\  r_{c(\omega)} \geq r} } B(c(\omega),\kappa r_{c(\omega)}^{\petal/(1+\petal)} r^{1/(1+\petal)})
\end{equation*}
with multiplicity $\lesssim 1$.  The fact that this is a cover is \cite[Theorem 3.1 (ii)]{SU2} and the fact that the multiplicity is bounded follows from \cite[Theorem 3.1 (i)]{SU2} which gives that the same collection of balls may be scaled by a uniform constant to provide a \emph{packing}.  In particular, for $R>r>0$ with $R$ sufficiently small, the set
\[\bigcup_{\substack{c(\omega) \in I \cap B(\xi,R) \\ R > r_{c(\omega)} \geq r} } B(c(\omega),\kappa r_{c(\omega)}^{\petal/(1+\petal)} r^{1/(1+\petal)})\]
has multiplicity $\lesssim 1$ and is contained in the ball $B(\xi,(\kappa+1)R)$. Therefore, 
\begin{align}
m(B(\xi,(\kappa+1)R)) \nonumber 
&\gtrsim  \sum_{\substack{c(\omega) \in I \cap B(\xi,R) \\ R > r_{c(\omega)} \geq r} } m\left(B(c(\omega),\kappa r_{c(\omega)}^{\petal/(1+\petal)} r^{1/(1+\petal)})\right) \nonumber \\
&\gtrsim  \sum_{\substack{c(\omega) \in I \cap B(\xi,R) \\ R > r_{c(\omega)} \geq r} } \left( r_{c(\omega)}^{\petal/(1+\petal)} r^{1/(1+\petal)}\right)^h \left(\frac{ r_{c(\omega)}^{\petal/(1+\petal)} r^{1/(1+\petal)}}{\rl}\right)^{(h-1)\petal} \nonumber \\
&\gtrsim  r^h \sum_{\substack{c(\omega) \in I \cap B(\xi,R) \\ R > r_{c(\omega)} \geq r} } \left(\frac{r_{c(\omega)}}{r}\right)^{\petal/(1+\petal)} 
\gtrsim   r^h \sum_{\substack{c(\omega) \in I \cap B(\xi,R) \\ R > r_{c(\omega)} \geq r} } 1 \label{estimate}
\end{align}
where the second inequality is an application of Theorem \ref{Global2}, and the third uses the fact that $r_{c(\omega)} \approx r_l$. Therefore
\begin{align*}
\sum_{\substack{c(\omega) \in I \cap B(\xi,R) \\ R > r_{c(\omega)} \geq r} } r_{c(\omega)}^h  &\leq \sum_{m \in \mathbb{Z}\cap[0,\text{\normalfont{log}}(R/r)]} \sum_{\substack{c(\omega) \in I \cap B(\xi,R) \\ e^{m+1}r > r_{c(\omega)} \geq e^{m}r} } r_{c(\omega)}^h \\
 &\lesssim \sum_{m \in \mathbb{Z}\cap[0,\text{\normalfont{log}}(R/r)]} \sum_{\substack{c(\omega) \in I \cap B(\xi,R) \\ e^{m+1}r > r_{c(\omega)} \geq e^{m}r} } r^h e^{mh}\\
 &\lesssim \sum_{m \in \mathbb{Z}\cap[0,\text{\normalfont{log}}(R/r)]}  r^h e^{mh} \sum_{\substack{c(\omega) \in I \cap B(\xi,R) \\ R > r_{c(\omega)} \geq e^{m}r} } 1 \\
 &\lesssim \sum_{m \in \mathbb{Z}\cap[0,\text{\normalfont{log}}(R/r)]}  r^h e^{mh} \left((r e^m)^{-h} m(B(\xi,(\kappa+1)R))\right) \quad  \text{by (\ref{estimate})}\\
&\lesssim \text{\normalfont{log}}(R/r) \ m(B(\xi,R))
\end{align*}
where the last inequality uses the fact that $m$ is a doubling measure.
\end{proof}
We also require the following  lemma. We omit the case $h=1$ for notational convenience. One may derive an appropriate analogue in that case as well but we will not need it.
\begin{lem}\label{Key}
Suppose $h \neq 1$. There exists a constant $C\geq 1$ such that, for $\xi \in J_r(T)$ and  $R>0$   sufficiently small and such that $\phi(\xi, R)$ is defined via $\omega \in \mathbf{\Omega}$ (recall discussion following Theorem \ref{Global2}),  there exists  $c(\omega) \in \cup_k T^{-k}(\omega) \subseteq  J_p(T)$ such that
 \[B(\xi,R) \subseteq B(c(\omega), C \phi(\xi,R)^{\frac{1}{(h-1)\petal}} r_{c(\omega)}).\]
\end{lem}
\begin{proof}
From \cite[(5.6)]{DU4} (noting that the $\phi$ used there is different from the $\phi$ used here), we get 
\begin{equation*}
B(\xi,R) \subseteq T_\xi^{-k}(B(\omega, C' \phi(\xi,R)^{\frac{1}{(h-1)\petal}}))
\end{equation*}
for  a uniform constant  $C'> 0$,  where $k$ is an appropriately chosen  integer and  $T_\xi^{-k}$ is an appropriately chosen holomorphic inverse branch of $T^k$ defined on  an appropriate  neighbourhood of $\omega$. The desired result then follows by setting   $c(\omega) := T_\xi^{-k}(\omega)$ and using    the Koebe Distortion theorem, noting that the canonical radius $r_{c(\omega)}$ satisfies
\begin{equation*} \label{nproperty}
r_{c(\omega)} : = |(T^{k})'(T^{-k}_\xi(\omega))| ^{-1}      = |(T_\xi^{-k})'(\omega)| 
\end{equation*}
 by the inverse function rule.  
\end{proof}

\subsection{Assouad dimension of $J(T)$: proof of Theorem \ref{julias} part \ref{thmj1}}\label{AsoJ}

  The lower bound will follow from our lower  bound for the Assouad spectrum of $J(T)$, see Section \ref{AsospecJ}.  Therefore we only need to prove the upper bound. That is, we show \[\aso J(T) \leq \max\{1,h\}.\]
Note that when $h \leq 1$,  $\aso J(T) \leq \aso m \leq 1,$ so throughout we assume that $h > 1$.

Let $\xi \in J(T)$, $\epsilon > 0$, and $R>r>0$ with $R/r \geq \max\{e^{\epsilon^{-1}},10\}$. Let $\{B(x_i,r)\}_{i \in X}$ be a centred $r$-packing of $B(\xi,R) \cap J(T)$ of maximal cardinality. We assume for convenience that $x_i \in J_r(T)$ for each $i$, which we may do since $J_r(T)$ is dense in $J(T)$.  This is not really necessary but allows for efficient application of Lemma \ref{Key}. 

If $z \in J_r(T)$ and $\rho>0$ is sufficiently small and $\phi(z, \rho)<1$, then  we may associate to the pair $(z,\rho)$ a point $c(\omega) \in J_p(T)$ coming from Lemma \ref{Key}.  That is, $\phi(z, \rho) $  is  defined via  $\omega$ (recall the discussion following Theorem \ref{Global2}) and Lemma \ref{Key} gives us the associated $c(\omega)$.   In particular, in this case $z$ belongs to the associated  canonical ball $B(c(\omega), r_{c(\omega)})$.  

Decompose $X$ as 
\[X = X_0 \cup X_1 \cup \bigcup\limits_{n=2}^{\infty} X_n\]
where
\begin{align*}
X_0 &= \{i \in X \mid  \ \text{if $(x_i,r)$ is associated with $c(\omega)$ then}\ r_{c(\omega)} \geq 5R\} \\
X_1 &=  \{i \in X \setminus X_0 \mid \phi(x_i,r) \geq (r/R)^\epsilon\}\\
\text{and} \ \ X_n &=  \{i \in X \setminus (X_0 \cup X_1) \mid e^{-n} \leq \phi(x_i,r) < e^{-(n-1)}\}.
\end{align*}
To study those $i \in X_0$,  we decompose $X_0$ further as 
\[X_0 = X_0^0 \cup \bigcup\limits_{n=1}^{\infty} X_0^n\]
where
\begin{align*}
X_0^0 &= \{i \in X_0 \mid \phi(x_i,r) \geq \phi(\xi,R)\}\\
X_0^n &= \{i \in X_0 \mid e^{-n}\phi(\xi,R) \leq \phi(x_i,r) < e^{-(n-1)} \phi(\xi,R)\}.
\end{align*}
By Theorem \ref{Global2}
\begin{align*}
R^h \phi(\xi,R)  \gtrsim m(B(\xi,R)) 
\gtrsim m\left(\bigcup_{i \in X_0^0} B(x_i,r)\right) 
\gtrsim \min_{i \in X_0^0} (\# X_0^0)  r^h \phi(x_i,r)\
\geq (\# X_0^0)  r^h \phi(\xi,R),
\end{align*}
which implies that \[\# X_0^0 \lesssim \left(\frac{R}{r}\right)^h.\]

Turning our attention to $X_0^n$ with $n \geq 1$, write $X_0^n(c(\omega))$ to denote the set of all $i \in X_0^n$ for which the pair $(x_i,r)$ is  associated with $c(\omega)   \in J_p(T)$.  In particular, Lemma \ref{Key} ensures that 
\[
B(x_i,r) \subseteq B(c(\omega), C (\phi(\xi,R) e^{-(n-1)})^{\frac{1}{(h-1)\petal}} r_{c(\omega)})
\]
for all $i \in X_0^n(c(\omega))$, where $C \geq 1$ is the constant coming from Lemma \ref{Key}.   Temporarily fix  $c(\omega) \in J_p(T)$ such that $X_0^n(c(\omega)) \neq \emptyset$.  We show that
 \begin{equation} \label{cprime}
 \size{c(\omega) - z}   \approx   \phi(z, \rho )^{\frac{1}{(h-1)\petal}} r_{c(\omega)} 
\end{equation}
for all $z \in J_r(T)$ and $\rho >0$  such that $ (z,\rho)$ is associated with $c(\omega)$.  This  is essentially proved in \cite{DU4}.  More precisely,   \cite[inequalities  (5.3) and  (5.5)]{DU4} give 
\[ \phi(z ,\rho)^{\frac{1}{(h-1)\petal}}\approx  \vert T^{k}(z) - \omega \vert
\]
where $k \in \mathbb{N}$ is as in the proof of Lemma \ref{Key}, that is, such that $T^k(c(\omega)) = \omega$ and $|(T^{k})'(c(\omega))| =  r_{c(\omega)}^{-1}$ (we note that the $\phi$ used in \cite{DU4} is different than the notation we are using, so we have translated the equation into our notation which is consistent with \cite{SU1},\cite{SU2}).  Then, applying the Koebe Distortion Theorem, 
\[
\vert T^{k}(z) - \omega \vert  \approx   r_{c(\omega)}^{-1}  \vert z-c(\omega) \vert \]
establishing \eqref{cprime}.

Suppose $i \in X_0^N(c(\omega))$ for some large $N$, which implies $\phi(x_i,r) \leq e^{-(N-1)}\phi(\xi, R)$.   Recall that  $r_{c(\omega)} \geq 5R$.  Let $j \in \mathbb{N}$ be such that  $n_{j+1}=n_{j+1}(x_i) \geq n_{j }(x_i) =n_{j} $ and 
\[
|(T^{n_{j+1}})'(x_i)|^{-1} = r_{j+1}(x_i) \leq r < r_j(x_i)
\]
 and  $T^{k}(x_i) \in U_\omega$ for all $n_j<k<n_{j+1}$.  For $m$ chosen as $k$   in Lemma \ref{Key},  $T^m(c(\omega)) = \omega$ and $|(T^{m})'(c(\omega))| = r_{c(\omega)}^{-1}$, we have $T^m(\xi) \in U_\omega$. Note that 
\[
|(T^{n_{j+1}})'(x_i)|^{-1} = r_{j+1}(x_i)\leq  r <R\leq r_{c(\omega)}/5 \approx |(T^{m})'(c(\omega))|^{-1},
\]
and  $T^{k}(\xi) \in U_\omega$ for all $n_j <  m \leq k \leq  l \leq n_{j+1}$ for some $l$ satisfying $|(T^{l})'(\xi)|^{-1} \gtrsim R$. It follows that   $(\xi, aR)$ is also associated with  $c(\omega)$  for some $a \approx 1$.   Then by \eqref{cprime} 
\begin{align*}
|\xi -c(\omega)|   \approx    \phi(\xi,aR)^{\frac{1}{(h-1)\petal}} r_{c(\omega)} 
\approx    \phi(\xi,R)^{\frac{1}{(h-1)\petal}} r_{c(\omega)}
\end{align*}
since $a \approx 1$, and
\begin{align*}
 |x_i - c(\omega)| \approx   \phi(x_i, r )^{\frac{1}{(h-1)\petal}} r_{c(\omega)} 
 \lesssim  e^{-\frac{N}{(h-1) \petal} }\phi(\xi,R)^{\frac{1}{(h-1)\petal}} r_{c(\omega)}
\end{align*}
and therefore  (see Figure \ref{boundRfig})
\begin{align}  \label{boundingk}
R \geq |\xi - x_i| \geq  |\xi -c(\omega)| - |x_i - c(\omega)|  
 \gtrsim  \phi(\xi,R)^{\frac{1}{(h-1)\petal}} r_{c(\omega)}
\end{align}
for   $N$ chosen large enough depending only on various  implicit  constants.    We fix such $N$ in the following discussion.

\begin{figure}[H] 
\centering
\begin{tikzpicture}
\filldraw[black] (0.2,0.2) circle (2pt);
\filldraw[black] (2,1.5) circle (2pt);
\filldraw[black] (-3,-2) circle (2pt);
\draw[dashed] (2,1.5) -- (3.84,-0.34);
\node at (0.5,0.2) {$x_i$};
\node at (2.3,1.6) {$\xi$};
\node at (-3.4,-2.37) {$c(\omega)$};
\node at (2.8,0.2) {$R$};
\draw (2,1.5) circle (2.6cm);
\draw [decorate,decoration={brace,amplitude=10pt},xshift=-4pt,yshift=0pt]
(-2.87,-1.8) -- (2,1.6)node [black,midway,xshift=-58pt,yshift=21pt] {
$\approx   \phi(\xi,R)^{\frac{1}{(h-1)\petal}} r_{c(\omega)}$};
\draw [decorate,decoration={brace,mirror,amplitude=10pt},xshift=-4pt,yshift=0pt]
(-2.75,-2.15) -- (0.36,0.05)node [black,midway,xshift=80pt,yshift=-20pt] {
$ \lesssim   e^{-\frac{N}{(h-1) \petal} }\phi(\xi,R)^{\frac{1}{(h-1)\petal}} r_{c(\omega)}$};
\end{tikzpicture}
\caption{Bounding $R$ from below.}

\label{boundRfig}

\end{figure}
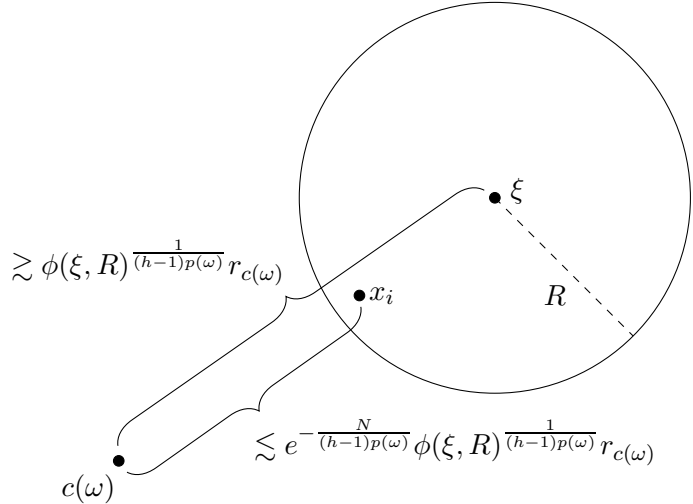

We may assume $X_0^n(c(\omega)) \neq \emptyset$ for some $n \geq N$, since otherwise $\phi(x_i, r) \gtrsim \phi(\xi, R)$ for all $i \in X_0$ and the argument bounding $\# X_0^0$ also applies to bound $\# X_0$.  By Theorem \ref{Global2} and Lemma \ref{Key}
\begin{align*}
m\left( \bigcup\limits_{i \in X_0^n(c(\omega))} B(x_i,r)  \right) & \lesssim m \left( B(c(\omega), C (\phi(\xi,R) e^{-(n-1)})^{\frac{1}{(h-1)\petal}} r_{c(\omega)}) \right) \\
& \approx (e^{-n} \phi(\xi,R))^{\frac{h}{(h-1)\petal}}  r_{c(\omega)}^h \phi\left(c(\omega), C (\phi(\xi,R) e^{-(n-1)})^{\frac{1}{(h-1)\petal}} r_{c(\omega)})  \right)\\
& \approx (e^{-n} \phi(\xi,R))^{\frac{h}{(h-1)\petal}}  r_{c(\omega)}^h \left(\frac{ C (\phi(\xi,R) e^{-(n-1)})^{\frac{1}{(h-1)\petal}} r_{c(\omega)}}{r_{c(\omega)}}  \right)^{(h-1)p(\omega)}\\
& \approx (e^{-n} \phi(\xi,R))^{\frac{h}{(h-1)\petal}}  r_{c(\omega)}^h e^{-n} \phi(\xi,R)
\end{align*}
where $C$ is the constant from Lemma \ref{Key}. In the other direction, as $\{x_i\}_{i \in X_0^n(c(\omega))}$ is an $r$-packing,
\begin{align*}
m\left( \bigcup\limits_{i \in X_0^n(c(\omega))}B(x_i,r)  \right)   \geq \sum_{i \in X_0^n(c(\omega))}  m\left(B(x_i,r)\right)  \gtrsim \left(\# X_0^n(c(\omega))\right)  r^h e^{-n} \phi(\xi,R)
\end{align*}
and so
\begin{align*}
\# X_0^n(c(\omega))  \lesssim  (e^{-n} \phi(\xi,R))^{\frac{h}{(h-1)\petal}}  r_{c(\omega)}^h r^{-h} 
\lesssim e^{\frac{-nh}{(h-1)\pmax}} \left(\frac{R}{r} \right)^h
\end{align*}
by \eqref{boundingk}.  The  number of distinct $c(\omega)$ giving rise to non-empty $X_0^n(c(\omega))$ with $n \geq N$  is $\lesssim 1$ since  $(\xi, aR)$ is associated with  each such $c(\omega)$ for some $a \approx 1$. Therefore
\begin{equation*}
\# X_0^n  \lesssim e^{\frac{-nh}{(h-1)\pmax}} \left(\frac{R}{r} \right)^h.
\end{equation*}
Pulling these estimates together, we get
\begin{align}\label{estX0}
\# X_0   = \# X_0^0 + \sum\limits_{n=1}^{\infty} \#X_0^n 
\lesssim \left(\frac{R}{r}\right)^h +  \sum\limits_{n=1}^{\infty} e^{\frac{-nh}{(h-1)\pmax}} \left(\frac{R}{r} \right)^h 
 \lesssim  \left(\frac{R}{r}\right)^h.
\end{align}
If $i \in X_1$, then
\begin{align*}
R^h \geq R^h \phi(\xi,R) \gtrsim m(B(\xi,R)) 
\gtrsim \min_{i \in X_1} (\# X_1)  r^h \phi(x_i,r) 
\gtrsim (\# X_1)  r^h \left(\frac{r}{R}\right)^\epsilon
\end{align*}
which proves
\begin{equation}\label{estX1}
\# X_1 \lesssim \left(\frac{R}{r}\right)^{h+\epsilon}.
\end{equation}
Finally, we turn our attention to $X_n$. If $i \in X_n$ for $n \geq 2$, then $\phi(x_i,r) < e^{-(n-1)}$, and therefore by Lemma \ref{Key} the ball $B(x_i,r)$ is contained in the squeezed canonical ball 
\[B(c(\omega),C e^{\frac{-(n-1)}{(h-1)\petal}} r_{c(\omega)})\]
for some $c(\omega) \in J_p(T)$ with $r_{c(\omega)} < 5R$. Therefore, $ r/C \leq r_{c(\omega)} < 5R < 6CR$ and, noting that $h > 1$, 
\begin{align*}
|c(\omega) - \xi| \leq |c(\omega)-x_i| + |x_i - \xi| 
&\leq C r_{c(\omega)} +R
 < 5CR+R 
\leq 6CR
\end{align*}
and so $c(\omega) \in B(\xi,6CR)$.  For $p \in \{1,\dots,\pmax\}$, let
\[X_n^p = \{i \in X_n \mid \text{$(x_i,r)$ is associated with $c(\omega)$ and $p(\omega) = p$}\}\] and let \[I_p = \bigcup_{\substack{\omega \in \mathbf{\Omega} \\ \petal = p}} I\]
where $I$ is defined in the same way as in Lemma \ref{CountCanon}. 
Then
\begin{align*}
m\left(\bigcup\limits_{i \in X_n^p} B(x_i,r)\right) &\leq m\left (\bigcup_{\substack{c(\omega) \in I_p \cap B(\xi,6CR) \\ 6CR > r_{c(\omega)} \geq r/C} } B\left(c(\omega),C e^{\frac{-(n-1)}{(h-1)p}} r_{c(\omega)}\right) \right)\\
&\leq  \sum_{\substack{c(\omega) \in I_p \cap B(\xi,6CR) \\ 6CR > r_{c(\omega)} \geq r/C} }  m\left( B\left(c(\omega), C e^{\frac{-(n-1)}{(h-1)p}} r_{c(\omega)}\right) \right)\\
&\lesssim  \sum_{\substack{c(\omega) \in I_p \cap B(\xi,6CR) \\ 6CR > r_{c(\omega)} \geq r/C} } e^{\frac{-nh}{(h-1)p}} r_{c(\omega)}^h \phi(c(\omega), C e^{\frac{-(n-1)}{(h-1)p}} r_{c(\omega)}) \\
&\lesssim \sum_{\substack{c(\omega) \in I_p \cap B(\xi,6CR) \\ 6CR > r_{c(\omega)} \geq r/C} } e^{\frac{-nh}{(h-1)p}} r_{c(\omega)}^h e^{-(n-1)} \qquad  \text{by Theorem \ref{Global2} (ii)}\\
&\lesssim e^{-n}  e^{\frac{-nh}{(h-1)p}} \left(\text{log}(R/r) + \text{log}(6C^2)\right) m(B(\xi,R)) \ \text{by Lemma \ref{CountCanon}}\\
&\lesssim e^{-n}  e^{\frac{-nh}{(h-1)p}} \text{log}(R/r)  R^h \phi(\xi,R) \\
&\lesssim e^{-n}  e^{\frac{-nh}{(h-1)p}} \epsilon^{-1} n R^h.
\end{align*}
The final line uses the estimate $(r/R)^\epsilon > e^{-n}$ which holds whenever  $X_n$ is non-empty.  On the other hand,  by Theorem \ref{Global2}
\begin{align*}
m\left(\bigcup\limits_{i \in X_n^p} B(x_i,r) \right) \geq \sum\limits_{i \in X_n^p} m(B(x_i,r)) 
\gtrsim ( \# X_n^p) r^h  e^{-n}.
\end{align*}
Therefore
\begin{align*}
\# X_n^p \lesssim \epsilon^{-1}n e^{\frac{-nh}{(h-1)p}} \left(\frac{R}{r}\right)^h 
\lesssim \epsilon^{-1}n  e^{\frac{-nh}{(h-1)\pmax}} \left(\frac{R}{r}\right)^h
\end{align*}
which gives
\begin{align}\label{estXn}
\# X_n \leq \sum\limits_{p=1}^{\pmax} \# X_n^p 
\lesssim \epsilon^{-1}n e^{\frac{-nh}{(h-1)\pmax}} \left(\frac{R}{r}\right)^h.
\end{align}
Combining (\ref{estX0}),(\ref{estX1}) and (\ref{estXn}), 
\begin{align*}
\# X   = \# X_0 + \# X_1 + \sum\limits_{n=2}^{\infty} \# X_n  
&\lesssim  \left(\frac{R}{r}\right)^h +  \left(\frac{R}{r}\right)^{h+\epsilon} +  \sum\limits_{n=2}^{\infty} \epsilon^{-1}n e^{\frac{-nh}{(h-1)\pmax}} \left(\frac{R}{r}\right)^h \\
&\lesssim  \left(\frac{R}{r}\right)^{h+\epsilon} + \epsilon^{-1}  \left(\frac{R}{r}\right)^h
\end{align*}
which proves that $\aso J(T) \leq h+\epsilon$, and letting $\epsilon \rightarrow 0$ proves that $\aso J(T) \leq h$, as required.

\subsection{Lower dimension of $J(T)$: proof of Theorem \ref{julias} part \ref{thmj2}}\label{LowJ}

The upper bound will follow from our upper  bound for the lower spectrum of $J(T)$, see Section \ref{LowspecJ}.  Therefore we only need to prove the lower bound. That is, we show \[\low J(T) \geq \min\{1,h\}.\]
Note that when $h \geq 1$,  $\low J(T) \geq \low m \geq 1$ so we may assume throughout that $h < 1$.

Let $\xi \in J(T)$,  and $1>R>r>0$ with $R/r \geq 10$. Let $\{B(y_i,r)\}_{i \in Y}$ be a centred $r$-covering of $B(\xi,R) \cap J(T)$ of minimal cardinality. We assume for convenience that each $y_i \in J_r(T)$, which we may do since $J_r(T)$ is dense in $J(T)$.

As is standard by now, if  $\phi(y_i, r)>1$, then  we may associate to the pair $(y_i,r)$ a point $c(\omega) \in J_p(T)$ coming from Lemma \ref{Key}.  That is, $\phi(y_i, r) $  is  defined via  $\omega$ (recall the discussion following Theorem \ref{Global2}) and Lemma \ref{Key} gives us the associated $c(\omega)$.   In particular, in this case $y_i$ belongs to the associated  canonical ball $B(c(\omega), r_{c(\omega)})$.

Decompose $Y$ as $Y = Y_0 \cup Y_1$ where
\begin{align*}
Y_0 &= \{i \in Y \mid \text{if $(y_i,r)$ is associated with $c(\omega)$ then  $r_{c(\omega)} \geq 5R$}\}\\
Y_1 &= Y \setminus Y_0.
\end{align*}
As $\{B(y_i,r)\}_{i \in Y}$ is a covering of $B(\xi,R) \cap J(T)$, 
\begin{align}\label{measure}
m(B(\xi,R))  \leq m\left(\cup_{i \in Y} B(y_i,r)\right) 
=  m\left(\cup_{i \in Y_0} B(y_i,r)\right) +  m\left(\cup_{i \in Y_1} B(y_i,r)\right) 
\end{align}
and therefore one of the terms in (\ref{measure}) must be at least $(m(\xi,R))/2$.

Suppose that the term involving $Y_0$ is at least $(m(\xi,R))/2$. Then we write
\[Y_0^0 = \{i \in Y_0 \mid \phi(y_i,r) \leq K \phi(\xi,R) \} \]
where $K>0$ is a constant chosen according to the following lemma.

\begin{lem}
We may choose $K>0$ independently of $R$ and $r$ sufficiently large  such that
\begin{equation} \label{choosingk}
\frac{m\left(\cup_{i \in Y_0 \setminus Y_0^0} B(y_i,r) \right)}{m(B(\xi,R))} \leq \frac{1}{100}. 
\end{equation}
\end{lem}
\begin{proof}
Write $Y(c(\omega))$ to denote the set of all $i \in Y_0 \setminus Y_0^0$ such that $(y_i,r)$ is associated with $c(\omega)$.  Then
\begin{align} \label{yballin}
B(y_i,r) \subseteq B(c(\omega), C \phi(y_i,r)^{\frac{1}{(h-1)\petal}} r_{c(\omega)})  
\subseteq B(c(\omega), C   (K\phi(\xi,R))^{\frac{1}{(h-1)\petal}} r_{c(\omega)})
\end{align}
using the definition of $Y_0 \setminus Y_0^0$ and where $C \geq 1$ is the constant from Lemma \ref{Key}.  Consider non-empty $Y(c(\omega))$.   Since $r_{c(\omega)} \geq 5R$ we may follow the proof of   \eqref{boundingk} to show that $(\xi,aR)$ is also associated with $c(\omega)$  for some $a \approx 1$ and that $K$ can be chosen large enough such that  $|c(\omega) - \xi| \lesssim R$.  Then, since $m$ is doubling, by Theorem \ref{Global2} 
\begin{align*}
R^h \phi(\xi, R) \approx m(B(\xi,R)) 
\approx m(B(c(\omega), R)) 
\approx R^h \phi(c(\omega),R) 
\approx R^h \left(\frac{R}{r_{c(\omega)}} \right)^{(h-1)p(\omega)}
\end{align*}
and so
\begin{equation} \label{boxofrain}
r_{c(\omega)} \phi(\xi,R)^{\frac{1}{(h-1)\petal}} \approx R.
\end{equation}
Applying Theorem \ref{Global2} (ii) and \eqref{yballin},
\begin{align*}
 \frac{m\left(\cup_{i \in Y(c(\omega))} B(y_i,r) \right)}{m(B(\xi,R))} &\lesssim \frac{ \left( (K\phi(\xi,R))^{\frac{1}{(h-1)\petal}} \right)^{h+(h-1)\petal} r_{c(\omega)}^h }{\phi(\xi,R) R^h} \\
&=   K^{1+\frac{h}{(h-1)\petal}} \left(\frac{\phi(\xi,R)^{\frac{1}{(h-1)\petal}} r_{c(\omega)}}{R} \right)^h 
\lesssim   K^{1+\frac{h}{(h-1)\petal}}  
\end{align*}
by \eqref{boxofrain}.  Note that    the  number of distinct squeezed canonical balls giving rise to non-empty $Y(c(\omega))$ is $\lesssim 1$.  This is because $\phi(\xi,aR)$ is associated with each such  $c(\omega)$ for some $a \approx 1$.  Using the general bound $h > \pmax/(1+\pmax)$, we see $1+h/((h-1)\petal) < 0$, and therefore we may  choose $K$ large enough to ensure  \eqref{choosingk}.  
\end{proof}

Applying \eqref{choosingk}, 
\begin{align*}m\left( \cup_{i \in Y_0^0} B(y_i,r) \right) \approx m\left( \cup_{i \in Y_0} B(y_i,r) \right) 
\geq m(B(\xi,R))/2 \end{align*}
which gives
\begin{align*}
R^h \phi(\xi,R) \lesssim m(B(\xi,R)) \lesssim m\left(\cup_{i \in Y_0^0} B(y_i,r) \right) 
\lesssim (\# Y_0^0) r^h \phi(\xi,R)
\end{align*}
where the last inequality uses the definition of $Y_0^0$. Therefore
\begin{equation}\label{estY0}
\# Y_0 \geq \# Y_0^0 \gtrsim \left(\frac{R}{r} \right)^h.
\end{equation}
Now, suppose that the second term of (\ref{measure}) is at least $m(B(\xi,R))/2$. Let $\epsilon>0$ and  write
\[Y_1^0 = \left\{ i \in Y_1 \mid \phi(y_i,r) \leq \left(\frac{R}{r} \right)^\epsilon \right\}.\]
If $i \in Y_1 \setminus Y_1^0$, then this implies that $\phi(y_i,r) >  \left({R}/{r} \right)^\epsilon$, and therefore by Lemma \ref{Key} the ball $B(y_i,r)$ is contained in the squeezed canonical ball 
\[B\left(c(\omega),  C \left(\frac{R}{r} \right)^{\frac{\epsilon}{(h-1)\petal}} r_{c(\omega)}\right)\]
for some $c(\omega) \in J_p(T)$.   Therefore, recalling the definition of $Y_1$, $ r/C \leq r_{c(\omega)} < 5R < 6CR$ and, using  $h \leq 1$, 
\[
|c(\omega) - \xi| \leq |c(\omega)-y_i| + |y_i - \xi|
 \leq C r_{c(\omega)} +R 
< 5CR+R 
\leq 6CR
\]
and so $c(\omega) \in B(\xi,6CR)$.  Therefore 
\begin{align*}
m\left(\cup_{i \in Y_1 \setminus Y_1^0} B(y_i,r) \right) 
&\lesssim \sum_{\substack{c(\omega) \in J_p(T) \cap B(\xi,6CR) \\ 6CR > r_{c(\omega)} \geq r/C} } m\left(B\left(c(\omega),  C \left(\frac{R}{r} \right)^{\frac{\epsilon}{(h-1)\petal}} r_{c(\omega)}\right) \right) \\
&\lesssim  \sum_{\substack{c(\omega) \in J_p(T) \cap B(\xi,6CR) \\ 6CR > r_{c(\omega)} \geq r/C} } \left(\frac{R}{r} \right)^{\frac{\epsilon h}{(h-1)\petal}} r_{c(\omega)}^h
\phi\left(c(\omega), \left(\frac{R}{r} \right)^{\frac{\epsilon}{(h-1)\petal}} r_{c(\omega)}\right) \\
&\lesssim  \sum_{\substack{c(\omega) \in J_p(T) \cap B(\xi,6CR) \\ 6CR> r_{c(\omega)} \geq r/C} } \left(\frac{R}{r} \right)^{\frac{\epsilon h}{(h-1)\petal}} r_{c(\omega)}^h
 \left(\frac{R}{r} \right)^{\epsilon} \qquad \text{(since $c(\omega) \in J_p(T)$)}\\
&\lesssim \left(\text{log}(R/r) + \text{log}(6C^2)\right)  \left(\frac{R}{r} \right)^{\frac{\epsilon (h+(h-1)\pmax)}{(h-1)\pmax}} m(\xi,R)
\end{align*}
where the last inequality uses Lemma \ref{CountCanon} and the fact that $m$ is doubling. Note that the exponent of the term involving $R/r$ is negative, recalling that $h > {\pmax}/{(1+\pmax)}$, so for sufficiently large $R/r$, balls with centres in $Y_1 \setminus Y_1^0$ cannot carry a fixed proportion of $m(B(\xi,R))$, and so
\begin{align*}m\left( \cup_{i \in Y_1^0} B(y_i,r) \right) \approx m\left( \cup_{i \in Y_1} B(y_i,r) \right) 
\geq m(B(\xi,R))/2. 
\end{align*}
Therefore, 
\begin{align*}
R^h \lesssim R^h \phi(\xi,R) \approx m(B(\xi,R)) 
\lesssim m\left( \cup_{i \in Y_1^0} B(y_i,r) \right) 
\lesssim ( \# Y_1^0)  r^h \left(\frac{R}{r} \right)^\epsilon
\end{align*}
where the last inequality uses the definition of $Y_1^0$. Therefore
\begin{align}\label{estY1}
\# Y_1  \geq \# Y_1^0 
\gtrsim \left(\frac{R}{r} \right)^{h-\epsilon} .
\end{align}
We have proven that at least one of (\ref{estY0}) and (\ref{estY1}) must hold, and therefore
\[ \# Y = \# Y_0 + \# Y_1 \gtrsim \left(\frac{R}{r} \right)^{h-\epsilon} \]
which proves that $\low J(T) \geq h-\epsilon$, and letting $\epsilon \rightarrow 0$ proves the desired result.

\subsection{Assouad spectrum of $J(T)$: proof of Theorem \ref{julias} part \ref{thmj3}}\label{AsospecJ}

\subsubsection{When $h \leq 1$}\label{AsospecJ1}
 We show \[\asospec J(T) = h+\min\left\{1,\frac{\theta\pmax}{1-\theta}\right\}(1-h).\]
The upper bound follows from 
\[\asospec J(T) \leq \asospec m \leq h+\min\left\{1,\frac{\theta\pmax}{1-\theta}\right\}(1-h),\]
and for the lower bound, we can apply the following proposition (see \cite[Theorem 3.4.8]{Fr2}). 
\begin{prop}\label{rho}
Let $F \subset \mathr^n$ be bounded and suppose that \[\rho =\inf \normalfont{\{\theta \in (0,1) \mid \asospec F = \aso F\}}\] exists and $\rho \in (0,1)$ and $ \normalfont{\low F = \ubox F}$. Then for $\theta \in (0,\rho)$,
\begin{equation*}
 \normalfont{\asospec F \geq \ubox F + \frac{(1-\rho)\theta}{(1-\theta)\rho}(\aso F-\ubox F)}.
\end{equation*}
\end{prop}
Since $h \leq 1$, $\low J(T) = \ubox J(T) = h$ and the upper bound for $\asospec J(T)$ shows that $\rho \geq {1}/{(1+\pmax)}$, so we need only show that $\rho \leq  1/(1+\pmax)$. To do this, let $\omega \in \mathbf{\Omega}$ be such that $\petal = \pmax$, and recall that there exists some sufficiently small neighbourhood $U_\omega = B(\omega,r_{\omega})$ such that there exists a unique holomorphic inverse branch $T_{\omega}^{-1}$ of $T$ on $U_\omega$ such that $T_{\omega}^{-1}(\omega) = \omega$. Using  \cite[Lemma 1]{DU3}, for $\xi \in U_\omega \cap \big(J(T) \setminus \{\omega\}\big)$ and $n \in \mathbb{N}$, 
\[ \vert T_\omega ^{-n} (\xi) - \omega \vert \approx n^{-1/\pmax}.\]
Therefore, using $T$-invariance of $J(T)$, we can find a subset of $J(T)$ with the same Assouad spectrum as the   set $\{ n^{-1/\pmax}: n \in \mathbb{N} \}$.  Therefore, 
\begin{align*}
\asospec J(T) &\geq \asospec \{ n^{-1/\pmax}: n \in \mathbb{N} \} = \min\left\{1,\frac{\pmax}{(1+\pmax)(1-\theta)}\right\}
\end{align*}
by \cite[Corollary 6.4]{FYu}.  This is enough to ensure $\rho \leq  {1}/{(1+\pmax)}$. Therefore, by Proposition \ref{rho}, for $\theta \in (0,{1}/{(1+\pmax)})$ 
\begin{align*}\asospec J(T) \geq h + \frac{(1-1/(1+\pmax))\theta}{(1-\theta)/(1+\pmax)}(1-h) 
=  h+\frac{\theta\pmax}{1-\theta}(1-h)
\end{align*}
as required.

\subsubsection{When $h \geq 1$}\label{AsospecJ2}
 We show $\asospec J(T) = h.$ This follows easily, since $h=\ubox J(T) \leq \asospec J(T) \leq \aso J(T) = h.$

\subsection{Lower spectrum of $J(T)$: proof of Theorem \ref{julias} part \ref{thmj4}}\label{LowspecJ}

\subsubsection{When $h \leq 1$}\label{LowspecJ1}
 We show $\lowspec J(T) = h.$ This follows easily, since $h=\low J(T) \leq \lowspec J(T) \leq \lbox J(T) = h.$
\subsubsection{When $h \geq 1$}\label{LowspecJ2}
 We show \[\lowspec J(T) =  h+\min\left\{1,\frac{\theta\pmax}{1-\theta}\right\}(1-h).\]
Note that 
\[\lowspec J(T) \geq \lowspec m \geq   h+\min\left\{1,\frac{\theta\pmax}{1-\theta}\right\}(1-h)\]
and so it suffices to prove the upper bound. To do this, we require the following technical lemma, which is a quantitative version of the well-known Leau--Fatou flower theorem (see \cite[325--363]{shish} and \cite{milnor}). The quantitative version was not known to us initially and we could not find it in the literature, but it seems to be known in the complex dynamics community. We thank Davoud Cheraghi for explaining this version to us. 
 We note that the non-quantitative version, e.g. that stated in \cite{ADU, DU4}, is enough to bound the lower dimension from above, but not the lower spectrum.  
\begin{lem}\label{qflower}
Let $\omega \in \mathbf{\Omega}$ be a parabolic fixed point with petal number $\petal$. Then there exists a constant $C>0$ such that for all sufficiently small $r>0$, $B(\omega,r) \cap J(T)$ is contained in a $C\log(1/r)r^{1+\petal}$-neighbourhood of the set of  lines emanating from $\omega$ in the ($\petal$ many) repelling directions.
\end{lem}
\begin{proof}
  We may assume via standard reductions that $\omega = 0$ and that the repelling directions are $e^{n2\pi i/\petal}$ for $n = 0,1,\dots,\petal-1$. By the (non-quantitative) Leau--Fatou flower theorem,  $B(0,r) \cap J(T)$ is contained in a cuspidal neighbourhood of the repelling directions.  Moreover, we may assume (by a simple change of coordinates if necessary) that
\[
T(z) = z+az^{\petal+1} + O(z^{\petal+2} )
\]
for a strictly positive real coefficient $a$.  Apply  the coordinate transformation $\psi: z \mapsto 1/z^{\petal}$ which sends the fixed point to infinity and the  repelling directions   to 1.  We will study the neighbourhoods of the repelling directions by considering the dynamics of  the  conjugated map  $\psi \circ T \circ \psi^{-1}$  at infinity.  This map is multiply defined, with one branch for each repelling direction.  We will explicitly describe the principal  branch which sends $[0,\infty]$ to $[0,\infty]$.  The others are handled similarly.  For this branch,   $\psi \circ T \circ \psi^{-1}$  is a well-defined function and is approximately a (real) translation.  More precisely, computing the  Laurent series at $z=\infty$,
\[
\psi \circ T \circ \psi^{-1} (z) = z-c+O(1/z^{1/p})
\] 
where $c$ is a strictly positive real constant (in fact $c=a \, \petal)$.  This ensures that for some $\varepsilon>0$, $\psi(J(T) \cap B(0,\varepsilon))$ is contained in a strip bounded between the curves
\[
x\pm f(x) i
\]
for an increasing concave function $f:[0,\infty) \to [0,\infty)$ which satisfies
\[
f(x)-f(x-c) \lesssim \textup{Im}\left((x+f(x)i)^{-1/p}\right) \lesssim x^{-1/p} \frac{f(x)}{x}
\]
for all $x \in [1,\infty)$. In particular, for $x \geq 1$, by rearranging,
\[
1-\frac{C''}{x^{1/p+1}} \leq \frac{f(x-c)}{f(x)}
\] 
for some constant $C''>0$ and therefore
\[
f(x) \leq  C' \log(x)
\]
for some constant $C'>0$. The pre-image of this strip under the coordinate transformation consists of cuspidal neighbourhoods of the repelling axes, and an easy calculation gives the desired result. Indeed,  considering only the pre-image which maps 1 to 1, expanding at $x=\infty$,
\[
\psi^{-1}(x\pm C'\log(x) i) = x^{-1/\petal}+i \, O\left( \log(x) (x^{-1/\petal})^{1+\petal} \right)
\]
and so, writing $r=x^{-1/\petal}$,  $J(T)$ is contained in a $C\log(1/r) r^{1+\petal}$-neighbourhood of the real axis.  The other pre-images are handled similarly. 
\end{proof}
We can use our work on the lower spectrum of $m$  to show that the exponent used in Lemma \ref{qflower} is sharp.  Again, this seems to be well-known in the complex dynamics community (even a stronger form of sharpness than we give) but we mention it since we provide a new approach.
\begin{cor}
In the case where $\omega \in \mathbf{\Omega}$ is of maximal rank and $h>1$, the expression $Cr^{1+\petal}$ in Lemma \ref{qflower} cannot be replaced by $Cr^{1+\petal+\epsilon}$ for any $\epsilon >0$.
\end{cor}
\begin{proof}
Suppose that such an $\epsilon$-improvement was possible for some $\omega \in \mathbf{\Omega}$ of maximal rank. Then, taking efficient $r$-covers of the  improved cuspidal neighbourhood of $B(\omega, r^{1+\pmax+\epsilon}) \cap  J(T)$,  we would  obtain $\lowspec J(T) \leq 1$ for all $\theta \geq {1}/{(1+\pmax+\epsilon)}$.  This  contradicts the lower bound for the lower spectrum of $m$, proved in Section \ref{LowspecM2}.
\end{proof}
We can now prove the upper bound. Note that when $\theta \geq {1}/{(1+\pmax)}$, we immediately have $\lowspec J(T) \leq 1$ by Lemma \ref{qflower}, so we assume that $\theta \in (0, {1}/{(1+\pmax)})$. Let $\omega \in \mathbf{\Omega}$ be such that $p(\omega) = \pmax$, and let $r>0$ be sufficiently small. Then we can estimate 
$N_{r}(B(\omega,r^\theta) \cap J(T))$
 by first covering $ B(\omega,r^\theta) \cap J(T)$ with balls of radius $r^{\theta(1+\pmax)}$, and then covering each of those balls with balls of radius $r$. Using the fact that $\aso J(T) = h$, 
\begin{align*}
N_{r}(B(\omega,r^\theta) \cap J(T)) &\lesssim N_{r^{\theta(1+\pmax)}}(B(\omega,r^\theta) \cap J(T)) \left(\frac{r^{\theta(1+\pmax)}}{r} \right)^h\\
&\lesssim \frac{r^\theta \log(1/r^\theta)}{r^{\theta(1+\pmax)}}  \left(\frac{r^{\theta(1+\pmax)}}{r} \right)^h \qquad  \text{by Lemma \ref{qflower}}\\
&= r^{-\theta\pmax + h(\theta(1+\pmax)-1)}   \log(1/r^\theta)\\
&= \left(r^{\theta-1}\right)^{h+\frac{\theta\pmax}{\theta-1}(h-1)}  \log(1/r^\theta)
\end{align*}
which proves that $
\lowspec J(T) \leq h+(\theta\pmax/(1-\theta))(1-h)
$
as required. 

\section{Cremer points: proof of Theorem \ref{cremer}} \label{cremerproof}

We assume without loss of generality that 0 is the Cremer fixed point. Then 
\begin{equation} \label{poly1}
T(z) = e^{2 \pi i \alpha} z +O(z^{2})
\end{equation}
in a neighbourhood of 0 for some $\alpha \notin \mathbb{Q}$. Without loss of generality we may assume \eqref{poly1} holds in the closed ball $B(0,1)$.  It follows that for sufficiently small $z$  and $n \in \mathbb{N}$ such that $\{z, T(z), \dots, T^{n-1}(z)\} \subseteq B(0,1)$ we have
\begin{equation} \label{poly}
T^n(z) = e^{2 \pi i n\alpha} z +5^{n-1} O(z^{2})
\end{equation}
where the implicit constant is the same as in \eqref{poly1}, in particular, independent of $n$ and $z$.  This can be proved by a simple induction.    Let $C \geq 1$ denote the implicit constant from \eqref{poly1} and suppose $|z| \leq 1/C$ and that  $\{z, T(z), \dots, T^{n-1}(z)\} \subseteq B(0,1)$.  Assume $n \geq 2$ and  that \eqref{poly} has been verified for $T^{n-1}$ (Inductive Hypothesis).  Then
\begin{align*}
\left\lvert T^n(z)-e^{2 \pi i n \alpha} z \right\rvert & \leq \left\lvert T^{n-1}(T(z))-e^{2 \pi i (n-1) \alpha} T(z) \right\rvert + \left\lvert e^{2 \pi i (n-1) \alpha} T(z) -e^{2 \pi i n \alpha} z\right\rvert \\
 & \leq  5^{n-2}C |T(z)|^2  + \left\lvert T(z) -e^{2 \pi i  \alpha} z\right\rvert  \qquad \text{(by Inductive Hypothesis)}\\
 & \leq 5^{n-2}C  \left(|z|+C|z|^2 \right)^2 + C|z|^2 \qquad \text{(by \eqref{poly1})}\\
 & = C|z|^2 \left(5^{n-2}(1+2C|z| +C^2|z|^2)+1 \right)   \\
& \leq 5^{n-1}C |z|^2\qquad \text{(using $C|z| \leq 1$)}
\end{align*}
proving \eqref{poly}.  For $\delta \in (0,1)$, we say a set $A \subseteq \mathbb{R}^d$ is $\delta$-\emph{dense} in a set $B \subseteq \mathbb{R}^d$  if $\sup_{b \in B} \inf_{a \in A} |a-b | \leq \delta$ and let 
\[
M(\alpha, \delta) = \min\Big\{ m : \{0,\alpha (\textup{mod}\ 1), \dots, m\alpha (\textup{mod}\ 1) \} \text{ is $\delta$-dense in $[0,1)$}\Big\}.
\]
Note that $M(\alpha, \delta)$ is finite for all $\delta \in (0,1)$ since $\alpha \notin \mathbb{Q}$ and that $M(\alpha, \delta) \to \infty$ as $\delta \to 0$.

To prove that $\aso J(T) = 2$, we show there exist  $0<r<R<1$ with $r/R$ arbitrarily small such that
  $J(T)$ is $(c r)$-dense in $B(0,R)$ for some uniform constant $c \geq1$.  Let $R \in (0,1/5)$ be small and choose $r$ depending on $R$ such that $r/R \to 0$ sufficiently slowly  as $R \to 0$ to ensure
\begin{equation} \label{growth}
5^{M(\alpha, r/R )} R^{2} \leq r.
\end{equation}
Note that this forces $r/R^{2} \to \infty$. It is clear that we can choose $r$ in this way but, to be concrete, set $r= \lambda R$ where 
\[
\lambda = \inf\{ \lambda'>0 : \lambda' \geq  5^{M(\alpha,\lambda')+1} R\}.
\]
Note that $\lambda \in (0,1]$ is well-defined since $1 \geq 5^{M(\alpha,1)+1}R = 5R$.  Moreover, for fixed $\lambda' \in (0,1]$, 
\[
5^{M(\alpha,\lambda')+1} R
\]
is decreasing and  converges  to 0 as $R \to 0$ which is enough to ensure that $\lambda \to 0$ as $R \to 0$.

   Let  $y \in  B(0,R)$ be arbitrary and choose $z \in J(T)$ with $|z|= |y| $.  We can choose $z$ in this way for sufficiently small $R$ using P\'erez-Marco's celebrated result that Cremer fixed points are contained in a  non-trivial connected component of the Julia set \cite[Theorem 1]{hedgehog}. However, we note that it is possible to adapt our  proof so as not to rely on this result by simply using that the Julia set is perfect and the Cremer point is indifferent (that is, we can choose $z$ with $|z|$ `close enough' to $|y|$).  By definition, for $r<|y|$, we can find  $n \in \mathbb{Z}$ with $0 \leq n-1 \leq M(\alpha, r/|y| ) \leq M(\alpha, r/R ) $ such that 
\begin{equation} \label{modsmall}
\big\lvert (\arg(z) + 2 \pi n \alpha)  (\textup{mod}\ 2 \pi) - \arg(y)  \big\rvert \leq  2 \pi r/|y|.
\end{equation}
Using \eqref{growth},
\[
|z|+5^{M(\alpha, r/R )}  C |z|^2 \leq R+5^{M(\alpha, r/R )}  C R^2 \leq R+Cr \leq 1
\]
for sufficiently small $R$ and so $\{z, T(z), \dots, T^{n-1}(z)\} \subseteq B(0,1)$ by iteratively applying \eqref{poly}.  Therefore,  by \eqref{poly}, \eqref{growth}, and \eqref{modsmall}
\begin{align} \label{argclose}
| \arg(T^n(z)) -\arg(y) | \nonumber \lesssim  \big\lvert (\arg(z) + 2 \pi n \alpha)  (\textup{mod}\ 2 \pi) - \arg(y)  \big\rvert    +    5^{n-1}|z| \nonumber &\lesssim r/|y| + 5^{M(\alpha, r/R )} R^{2}/|y| \nonumber  \\  &\lesssim r/|y|.
\end{align}
Moreover, by \eqref{poly} and \eqref{growth},
\begin{equation} \label{modclose}
\Big\lvert |T^n(z)| - |y| \Big\rvert  = \Big\lvert |T^n(z)| - |z| \Big\rvert \lesssim  5^{n-1}|z|^{2} \lesssim   5^{M(\alpha, r/R )} R^{2} \leq r.
\end{equation}
Together \eqref{argclose} and \eqref{modclose} yield
\[
|T^n(z) - y| \lesssim r
\]
and since the Julia set is $T$-invariant the result follows.

\begin{center} \textbf{Acknowledgements}
\end{center}
\vspace{-3mm}
The authors thank Davoud Cheraghi and Mariusz Urba\'nski for helpful discussions. JMF was financially supported by an \textit{EPSRC Standard Grant} (EP/R015104/1) and a \textit{Leverhulme Trust Research Project Grant} (RPG-2019-034). LS was financially supported by the University of St Andrews.

\bibliographystyle{apalike}
\addcontentsline{toc}{section}{References}
\bibliography{Julia}

@article{ADU,
  title={Ergodic theory for {M}arkov fibred systems and parabolic rational maps},
  author={Aaronson, J. and Denker, M. and Urba{\'n}ski, M.},
  journal={Trans. Amer. Math. Soc.},
  volume={337},
  number={2},
  pages={495--548},
  year={1993}
}

@article{AHRS,
  title={${L}^p \to {L}^q$ bounds for spherical maximal operators},
  author={Anderson, T. C. and Hughes, K. and Roos, J. and Seeger, A.},
  journal={Math. Z.},
  volume={297},
  number={3-4},
  pages={1057-1074},
year={2021}
}

@book{B2,
  title={Iteration of rational functions},
  author={Beardon, A. F.},
  year={1991},
  volume={132},
  series={Graduate Texts in Mathematics},
  publisher={Springer-Verlag, New York}
}

@book{BP,
  title={Fractals in probability and analysis},
  author={Bishop, C. J.  and Peres, Y.},
  year={2017},
  publisher={Cambridge University Press, Cambridge},
  series={Cambridge Stud. Adv. Math.},
  volume={162}
}

@article{BOW,
  title={Hausdorff dimension of quasi-circles},
  author={Bowen, R.},
  journal={Publ. Math. IHES},
  volume={50},
  pages={11--25},
  year={1979},
}

@article{DU1,
  title={Absolutely Continuous Invariant Measures for Expansive Rational Maps with Rationally Indifferent Periodic Points},
  author={Denker, M.  and Urba{\'n}ski, M.},
  journal={Forum Math.},
  volume={3},
  number={6},
  pages={561--579},
  year={1991}
}

@article{DU2,
  title={Hausdorff and conformal measures on {J}ulia sets with a rationally indifferent periodic point},
  author={Denker, M.  and Urba{\'n}ski, M.},
  journal={J. Lond. Math. Soc. (2)},
  volume={43},
  number={1},
  pages={107--118},
  year={1991},
  publisher={Narnia}
}

@article{DU3,
  title={Geometric measures for parabolic rational maps},
  author={Denker, M.  and Urba{\'n}ski, M.},
  journal={Erg. Th. Dyn. Syst.},
  volume={12},
  number={1},
  pages={53--66},
  year={1992},
  publisher={Cambridge University Press}
}

@article{DU4,
  title={The capacity of parabolic {J}ulia sets},
  author={Denker, M.  and Urba{\'n}ski, M.},
  journal={Math. Z.},
  volume={211},
  number={1},
  pages={73--86},
  year={1992},
  publisher={Springer}
}

@book{FK,
  title={Fractal geometry. Mathematical foundations and applications},
  author={Falconer, K. J.},
  year={2014},
  edition={$3^{\text{rd}}$},
  publisher={John Wiley \& Sons, Chichester}
}

@article{FFK,
  title={Minkowski dimension for measures},
  author={ Falconer, K. J. and Fraser, J. M. and K{\"a}enm{\"a}ki, A.},
  journal={Proc. Amer. Math. Soc.},
  volume={151},
  number={2},
  pages={779--794},
  year={2023}
}

@article{Fr1,
  title="{Regularity of {K}leinian limit sets and {P}atterson-{S}ullivan measures}",
  author={Fraser, J. M.},
  journal={Trans. Amer. Math. Soc.},
  volume={372},
  number={},
  pages={4977--5009},
  year={2019},
  publisher={Wiley Online Library}
}

@book{Fr2,
  title="{Assouad dimension and fractal geometry}",
  author={Fraser, J. M.},
  year={2020},
  publisher={Cambridge University Press, Cambridge},
  volume={222},
  series={Cambridge Tracts in Math.}
}

@article{FYu,
  title={New dimension spectra: finer information on scaling and homogeneity},
  author={Fraser, J. M. and Yu, H.},
  journal={Adv. Math.},
  volume={329},
  pages={273--328},
  year={2018\textcolor{white}{}}
}

@article {geomded,
    AUTHOR = {Fraser, Jonathan M. and Stuart, Liam},
     TITLE = {The {A}ssouad spectrum of {K}leinian limit sets and
              {P}atterson-{S}ullivan measure},
   JOURNAL = {Geom. Dedicata},
  FJOURNAL = {Geometriae Dedicata},
    VOLUME = {217},
      YEAR = {2023},
    NUMBER = {1},
     PAGES = {Paper No. 1, 32},
      ISSN = {0046-5755},
   MRCLASS = {28A78 (30F40 37F32 37F35)},
  MRNUMBER = {4493662},
MRREVIEWER = {Yan Mary He},
       DOI = {10.1007/s10711-022-00734-2},
       URL = {https://doi-org.ezproxy.st-andrews.ac.uk/10.1007/s10711-022-00734-2},
}

@article{Ka2,
  title={Dimensions, {W}hitney covers, and tubular neighborhoods},
  author={K{\"a}enm{\"a}ki, A. and Lehrb{\"a}ck, J. and Vuorinen, M.},
  journal={Indiana Univ. Math. J.},
  volume={62},
  number={6},
  pages={1861--1889},
  year={2013},
  publisher={JSTOR}
}

@article{LK,
  title={Assouad dimension: antifractal metrization, porous sets, and homogeneous measures},
  author={Luukkainen, J.},
  journal={J. Korean Math. Soc.},
  volume={35},
  number={1},
  pages={23--76},
  year={1998},
  publisher={Korean Mathematical Society}
}

@article{MM,
  title={Hausdorff dimension for horseshoes},
  author={Manning, A.  and McCluskey, H.},
  journal={Erg. Th. Dyn. Syst.},
  volume={3},
  number={2},
  pages={251--260},
  year={1983},
}

@article{MC1,
  title={The classification of conformal dynamical systems},
  author={McMullen, C. T.},
  journal={Current Developments in Mathematics},
  volume={1995},
  pages={323--360},
  year={1995},
  publisher={International Press of Boston}
}

@article{MC3,
  title={Hausdorff dimension and conformal dynamics {II}: Geometrically finite rational maps},
  author={McMullen, C. T.},
  journal={Comm. Math. Helv.},
  volume={75},
  number={4},
  pages={535--593},
  year={2000},
  publisher={Springer}
}

@article {mcmullensiegel,
    AUTHOR = {McMullen, C.  T.},
     TITLE = {Self-similarity of {S}iegel disks and {H}ausdorff dimension of
              {J}ulia sets},
   JOURNAL = {Acta Math.},
  FJOURNAL = {Acta Mathematica},
    VOLUME = {180},
      YEAR = {1998},
    NUMBER = {2},
     PAGES = {247--292},
      ISSN = {0001-5962},
   MRCLASS = {58F23 (28A78 28A80 30D05)},
       DOI = {10.1007/BF02392901},
       }

@book{R, 
 series={Tracts in Mathematics}, 
title={Dimensions, Embeddings, and Attractors},  
publisher={Cambridge University Press}, 
author={Robinson, J. C.},
volume={186},
 year={2011}
}

@article{RS,
  title={Spherical maximal functions and fractal dimensions of dilation sets},
  author={Roos, J.  and Seeger, A.},
  journal={Amer. J. Math.},
  volume={145},
  number={4},
  pages={1077--1110},
year={2023}
}

@article{SU1,
  title={The geometry of conformal measures for parabolic rational maps},
  author={Stratmann, B. O.  and Urba{\'n}ski, M.},
  journal={Math. Proc. Cambridge Philos. Soc.},
  volume={128},
  number={1},
  pages={141--156},
  year={2000},
  organization={Cambridge University Press}
}

@incollection {condynsys,
    AUTHOR = {Sullivan, Dennis},
     TITLE = {Conformal dynamical systems},
 BOOKTITLE = {Geometric dynamics},
    SERIES = {Lecture Notes in Math.},
    VOLUME = {1007},
     PAGES = {725--752},
 PUBLISHER = {Springer-Verlag, Berlin},
      YEAR = {1983}
}

@article{SU2,
  title={Jarn{\'\i}k and {J}ulia; a {D}iophantine analysis for parabolic rational maps for geometrically finite {K}leinian groups with parabolic elements},
  author={Stratmann, B. O.  and Urba{\'n}ski, M.},
  journal={Math. Scand.},
  volume={91},
  number={1},
  pages={27--54},
  year={2002},
  publisher={JSTOR}
}

@incollection{shish, 
series={London Math. Soc. Lecture Note Ser., \normalfont{Tan, L. (Ed.)}}, 
title={Bifurcation of parabolic fixed points}, 
booktitle={The Mandelbrot Set, theme and variations},
publisher={Cambridge University Press, Cambridge}, 
author={Shishikura, Mitsuhiro}, 
volume={274},
year={2000}, 
pages={325–363}, 
collection={London Math. Soc. Lecture Note Ser.},
}

@book {milnor,
    AUTHOR = {Milnor, J.},
     TITLE = {Dynamics in one complex variable},
    SERIES = {Ann. of Math. Stud.},
    VOLUME = {160},
   EDITION = {$3^{\text{rd}}$},
 PUBLISHER = {Princeton University Press, Princeton, NJ},
      YEAR = {2006},
     PAGES = {viii+304},
      ISBN = {978-0-691-12488-9; 0-691-12488-4},
   MRCLASS = {37Fxx (30-01 30D05 37-01)},
  MRNUMBER = {2193309},
}

@article{stuartsurvey,
  title={A new perspective on the {S}ullivan dictionary via {A}ssouad type dimensions and spectra},
  author={Fraser, J. M. and Stuart, Liam},
  journal={Bull. Amer. Math. Soc. (N.S.)},
  volume={61},
  number={1},
  pages={103-118},
  year={2024}
}

@article{cheraghi,
  title={Dimension paradox of irrationally indifferent attractors},
  author={Cheraghi, D. and DeZotti, A. and Yang, F.},
  journal={\normalfont{Preprint, available at: https://arxiv.org/abs/2003.12340}},
  year={2020}
}

@article{LG,
author = { Geyer, L.},
title = {Porosity of parabolic {J}ulia sets},
journal = {Complex Var. Theory Appl.},
volume = {39},
number = {3},
pages = {191-198},
year  = {1999},
publisher = {Taylor & Francis},
}

@article{hedgehog,
author = {P\'erez-Marco, Ricardo},
title = {Fixed points and circle maps},
journal = {Acta Math.},
volume = {179},
number={2},
pages = {243-294},
year  = {1997},
}

@book{noninvertible,
title = {Non--invertible dynamical systems. Vol. 3. Analytic endomorphisms of the {R}iemann sphere},
author = {Sara Munday and Mario Roy and Mariusz Urba\'nski},
volume={69.3},
series={De Gruyter Exp. Math.},
publisher = {De Gruyter, Berlin},
year = {2023},
}
\end{document}